%% file: main.tex
\documentclass[sigconf]{acmart}

\AtBeginDocument{%
  \providecommand\BibTeX{{%
    \normalfont B\kern-0.5em{\scshape i\kern-0.25em b}\kern-0.8em\TeX}}}

\acmConference[Performance 2020]{IFIP WG 7.3 Performance 2020}{November 02--06, 2020}{Milan, Italy}

\include{new_math_notation_command}
\usepackage[noend]{algpseudocode}
\usepackage{algorithm}
\newtheorem*{remark}{Remark}

\usepackage{enumitem}

\usepackage{graphicx}
\usepackage{caption}
\usepackage{subcaption}

\begin{document}

\title{Optimal Control in Fluid Models of n\textsf{x}n Input-Queued Switches\\ under Linear Fluid-Flow Costs}
\author{Yingdong Lu}
\email{yingdong@us.ibm.com}
\affiliation{%
IBM Research
}

\author{Mark S. Squillante}
\email{mss@us.ibm.com}
\affiliation{%
	IBM Research
}

\author{Tonghoon Suk}
\email{tonghoon.suk@gmail.com}
\affiliation{%
	IBM Research
}








\renewcommand{\shorttitle}{Optimal Control in Fluid Models of Input-Queued Switches under Linear Costs}
\renewcommand{\shortauthors}{author names}

\begin{abstract}
Most of the early input-queued switch research focused on establishing throughput optimality of the max-weight scheduling policy,
with some recent research showing that max-weight scheduling is optimal with respect to total expected delay asymptotically in the heavy-traffic regime.
However, the question of delay-optimal scheduling in input-queued switches remains open in general,
as does the question of delay-optimal scheduling under more general objective functions.
To gain fundamental insights into these very difficult problems, we consider a fluid model of $n \times n$ input-queued switches with associated fluid-flow costs,
and we derive an optimal scheduling control policy to an infinite horizon discounted control problem with a general linear objective function of fluid cost.
Our optimal policy coincides with the $c\mu$-rule in certain parameter domains.
More generally, due to the input-queued switch constraints, the optimal policy takes the form of the solution to a flow maximization problem,
after we identify the Lagrangian multipliers of some key constraints through carefully designed algorithms.
Computational experiments demonstrate the benefits of our optimal scheduling policy over variants of max-weight scheduling within fluid models of input-queued switches.
\end{abstract}





\settopmatter{printacmref=false}
\maketitle

\input{intro}

\input{model}

\input{optimal}

\input{proofs}

\input{numerical}

\section{Conclusions}\label{sec:conclusions}
%
We studied a fluid model of general $n \times n$ input-queued switches where each fluid flow has an associated cost,
and derived an optimal scheduling control policy under a general linear objective function based on minimizing discounted fluid cost over an infinite horizon.
We demonstrated that, while in certain parameter domains the optimal policy coincides with the $c\mu$-rule, in general the optimal policy is determined algorithmically
through a constrained flow maximization problem whose parameters, essentially Lagrangian multipliers of some key network constraints, were in turn identified by another set
of carefully designed algorithms.
Computational experiments within fluid models of input-queued switches demonstrated the significant benefits of our optimal scheduling policy over variants of  max-weight scheduling.

\bibliographystyle{ACM-Reference-Format}
\bibliography{fluid}

%

\end{document}

%% file: new_math_notation_command.tex
\newcommand{\bm}{\boldsymbol}

\newcommand{\brho}{\bm{\rho}}
\newcommand{\blambda}{\bm{\lambda}}

\newcommand{\bbeta}{\bm{\eta}}
\newcommand{\bzeta}{\bm{\zeta}}
\newcommand{\bmu}{\bm{\mu}}

\newcommand{\bnu}{\bm{\nu}}
\newcommand{\bomega}{\bm{\omega}}

\newcommand{\vQ}{\mathcal{Q}}
\newcommand{\bQ}{\boldsymbol{\mathcal{Q}}}
\newcommand{\vA}{\mathcal{A}}
\newcommand{\bA}{\boldsymbol{\mathcal{A}}}
\newcommand{\matA}{\bm{A}}

\newcommand{\bD}{\boldsymbol{\mathcal{D}}}
\newcommand{\vD}{\mathcal{D}}
\newcommand{\bq}{\bm{q}}
\newcommand{\bp}{\bm{p}}
\newcommand{\bc}{\bm{c}}
\newcommand{\bd}{\boldsymbol{\mathcal{\delta}}}
\newcommand{\vd}{{\delta}}
\newcommand{\bs}{\bm{s}}

\newcommand{\bw}{\bm{w}}

\newcommand{\RealSet}{\mathbb{R}}
\newcommand{\IntegerSet}{\mathbb{Z}}

\newcommand{\SetI}{\mathbb{I}}
\newcommand{\SetJ}{\mathbb{J}}

\newcommand{\SetT}{{\mathbb{T}}}
\newcommand{\SetU}{{\mathbb{U}}}

\newcommand{\bzero}{{\mathbf{0}}}

\newcommand{\E}{\mathbb{E}}

\newcommand{\bg}{\bm{g}}
\newcommand{\bbe}{\bm{e}}

\newcommand{\SetW}{{\mathbb{W}}}
\newcommand{\SetQ}{{\mathbb{Q}}}

%% file: intro.tex
\section{Introduction}
Input-queued switch architectures are widely used in modern computer and communication networks.
The optimal scheduling control of these high-speed, low-latency switch networks is critical for our understanding
of fundamental design and performance issues related to internet routers, cloud computing data centers, and high-performance computing.
A large and rich literature exists around optimal scheduling in these computer and communication systems.
This includes the extensive study of input-queued switches as an important mathematical model for a general class
of optimal control problems of broad interest in both theory and practice.

Most of the previous research related to scheduling control in input-queued switches has focused on throughput optimality.
In particular, the max-weight scheduling policy, first introduced in \cite{TasEph92} for wireless networks and subsequently in \cite{MCKEOWN96}
specifically for input-queued switches, is well-known to be throughput optimal.
The question of delay-optimal scheduling control in such switch networks, however, is far less clear with much more limited results.
This is due in large part because of the inherent difficulty of establishing delay (or equivalently, via Little's Law, queue length)
optimality for these types of stochastic systems in general.
Hence, previous research on optimal delay scheduling control in input-queued switches has focused on heavy-traffic and related asymptotic regimes;
see, e.g., \cite{Stolyar_cone_SSC, ShaWis_11, kang2009diffusion,shah2012optimal,zhong2014scaling}.

Such previous research includes showing that the max-weight scheduling policy is asymptotically optimal in heavy traffic for an objective function
of the summation of the squares of the queue lengths with the assumption of complete resource pooling~\cite{stolyar2004}.
Max-weight scheduling was then shown to be optimal in heavy traffic for an objective function of the summation of the queue lengths under the assumption
that all the ports are saturated~\cite{maguluri2016}.
This was subsequently extended to the case of incompletely saturated ports under the same objective function~\cite{maguluri2016b} and then to the case
of general linear objective functions~\cite{LuMaSq+18}.
Nevertheless, beyond these and related recent results limited to the heavy-traffic regime, the question of delay-optimal scheduling control in input-queued switches
remains open in general, as does the question of delay-optimal scheduling under more general objective functions.

In this paper, we seek to gain fundamental insights on optimal delay-cost scheduling in these stochastic systems by studying a fluid model of
general $n \times n$ input-queued switches where each fluid flow has an associated cost.
The objective of the corresponding optimal control problem is to determine the scheduling policy that minimizes the discounted summation over an infinite horizon
of general linear cost functions of the fluid levels associated with each queue.
Related research has been conducted in the queueing network literature; see, e.g., \cite{10.2307/171604,AvramBertsimasRicard95,maglaras2000,Bauerle2000}.
In particular, similar problems have been studied within the context of fluid models of multiclass queueing networks~\cite{AvramBertsimasRicard95,Bauerle2000}.
These previous studies take a classical optimal control approach based on exploiting Pontryagin's Maximum Principle~\cite{PoBoGa+62},
which itself only provides necessary conditions for optimality, to identify optimal policies.
However, while this framework enables with relative ease the derivation of optimal policies for fluid models of basic queueing networks,
the situation for input-queued switches is quite different and much more difficult.
Specifically, the highly constrained structure of the input-queued switch networks requires us to pay special attention to the feasibility of the optimal control problem.

To address these issues, we implicitly move the capacity constraint into the objective and identify the appropriate Lagrangian multiplier through carefully designed search algorithms.
Then, at any fluid level, we establish that the optimal scheduling policy is obtained through a solution to a flow maximization problem, which is also shown to be throughput optimal.
Our optimal policy coincides with the $c\mu$-rule in certain parameter domains.
These theoretical results reflect the high complexity nature of input-queued switches,
and are expected to be of interest more broadly than input-queued switch networks and more broadly than related classes of fluid models of stochastic networks with constraints.

We observe important differences in the decisions made under our optimal scheduling control policy in comparison with those made under
a cost-weighted variant of
the max-weight scheduling policy
and the $c\mu$-rule
within the fluid model of general $n \times n$ input-queued switches.
It is important to emphasize that our goal is to determine the optimal solution of the corresponding fluid control problem, which is at the core of the important
scheduling-decision differences between our optimal policy and those of the other scheduling policies.
Although we show that our flow maximization solution coincides with the $c\mu$-rule in some regions of the decision space, we also show that the $c\mu$-rule is not stable under
certain arrival rates and thus it cannot in general be the optimal scheduling policy.
In contrast to the max-weight scheduling policy which does not use any arrival rate information, we show that the optimal policy from our flow maximization solution for the $n \times n$
input-queued switch fluid control problem can depend in general on the arrival rates, which is consistent with known results established for the original (non-fluid limit) $2\times 2$
input-queued switch where the optimal policy takes into account the arrival processes in some regions of the decision space~\cite{LuMaSq+16}.
The cost-weighted max-weight scheduling policy has been shown to exhibit optimal queue-length scaling in the heavy traffic regime~\cite{LuMaSq+18},
suggesting that the importance of arrival-process information on the queue-length scaling of the optimal scheduling control policy tends to diminish asymptotically as the traffic intensity increases.

To further investigate these important differences, we conduct fluid-model computational experiments
with our optimal scheduling policy, the max-weight scheduling policy, and the $c\mu$-rule
to gain additional fundamental insights on various important theoretical issues with respect to optimal scheduling control in input-queued switch networks.
In comparisons with the max-weight scheduling policy, we find that our optimal scheduling control policy provides improvements of at least $10\%$ in most of the experiments,
sometimes rendering improvements of more than $50\%$.
Moreover, the improvements of our optimal policy over max-weight scheduling grow as the throughput increases. 
With respect to the $c\mu$-rule, we find that the comparisons with our optimal scheduling control policy fall into three different cases:
(1) The $c\mu$-rule coincides with the optimal policy, and thus is fluid-cost optimal;
(2) The $c\mu$-rule is unstable (not throughput optimal), and obviously not fluid-cost optimal;
(3) The $c\mu$-rule is stable, but not fluid-cost optimal.
Moreover, the greatest improvements observed for our optimal policy over stable $c\mu$-rule instances represent relative performance gaps of more than $70\%$.

The remainder of this paper is organized as follows.
Section~\ref{sec:model} presents our mathematical models, for both stochastic processes of input-queued switch networks and their mean-field limits, together with our formulation of the optimal scheduling control problems of interest.
Section~\ref{sec:control} then provides our analysis and results for optimal scheduling control and related theoretical properties, deferring our proofs until Section~\ref{sec:proofs}.
The results of computational experiments are presented in Section~\ref{sec:experiments}, followed by concluding remarks. 

%% file: model.tex
\section{Mathematical Models}\label{sec:model}
In this section, we first provide some technical preliminaries especially with respect to the notation used in the paper.
We then present a stochastic process model of general $n \times n$ input-queued switches, including the dynamics of queue lengths in discrete time.
Next, we introduce a sequence of such stochastic processes under an appropriate scaling and prove that every sample path of the sequence has a convergent
subsequence to deterministic processes in continuous time, i.e.,
our fluid models for general $n \times n$ input-queued switches;
this includes a characterization of admissible scheduling control policies for the fluid models.
Lastly, we present a formulation of the optimal scheduling control problems with the objective of finding an admissible policy that minimizes the
infinite-horizon discounted total linear cost of queue lengths in the fluid models.

\subsection{Technical Preliminaries}
Let $\RealSet$, $\RealSet_+$, $\RealSet^+$, $\IntegerSet$, $\IntegerSet_+$, and $\IntegerSet^+$ respectively denote the sets of real numbers, non-negative real numbers, positive real numbers,
integers, non-negative integers, and positive integers. 
For positive integer $n\in\IntegerSet^+$, we define $[n]:=\{1,2,\dots,n\}$ to be the set of all positive integers less than or equal to $n$.
The blackboard bold typefaces is used for general sets, e.g., $\SetI$ and $\SetJ$.
When the set $\SetI$ is finite, we represent its cardinality by $|\SetI|$; e.g., we have $|[n]|=n$ for $n\in\IntegerSet^+$.

We use the bold font to represent vectors, matrices, and real-valued functions on a finite set.
The function $\bmu:\SetI\to\RealSet$, defined on the finite set $\SetI$, can be considered as an $|\SetI|$-dimensional vector
$ \bmu~=~[\mu(\bs):\bs\in\SetI]$,
where $\mu(\bs)$ is the value of $\bmu$ at $\bs$.
We denote by $\RealSet^\SetI$ the set of all real-valued functions on $\SetI$. 
For finite sets $\SetI$ and $\SetJ$, $\RealSet^{\SetI\times\SetJ}$ is the set of all real-valued functions from $\SetI\times\SetJ$ in which an element
$\matA$ can also be represented by the matrix $\matA=[A(\bs,\brho):\bs\in\SetI, \brho\in\SetJ]$, where $A(\bs,\brho)$ is the value of the function
$\matA$ at $(\bs,\brho)\in\SetI\times\SetJ$.

For $\matA\in\RealSet^{\SetI\times\SetJ}$, $\bbeta\in\RealSet^\SetJ$, and $\bmu\in\RealSet^\SetI$,
we respectively define $\bmu \matA\in\RealSet^\SetJ$, $\matA\bq \in\RealSet^\SetI$, and $\mu A\eta\in \RealSet$ by 
\begin{align*}
    (\mu A)(\rho)&:=\sum_{\bs\in\SetI}\mu(\bs)A(\bs,\brho),\quad
    (A\eta)(\bs):=\sum_{\brho\in\SetJ} A(\bs,\brho)\eta(\brho),\\
     \mu A\eta&:=\sum_{\bs\in\SetI}\sum_{\brho\in\SetJ} \mu(\bs)A(\bs,\brho)\eta(\brho),
\end{align*}
which is similar to matrix-vector multiplication.
For $\bw,\bmu\in\RealSet^{\SetI}$, we also define $\bw\cdot\bmu\in\RealSet$ by $\bw\cdot\bmu:=\sum_{\bs\in\SetI}w(\bs)\mu(\bs)$,
which is the same as the inner-product of two vectors.
We denote the $1$-norm of a vector by $\|\cdot\|_1$, namely for $\bmu\in\RealSet^{\SetI}$, 
$\|\bmu\|_1~:=~\sum_{\bs\in\SetI}|\mu(\bs)|$.
Finally, we use the sans serif font for random variables and use the bold sans serif font for random vectors, e.g., $\vQ$ and $\bQ$, respectively.

\subsection{Stochastic Models}\label{sec:stochastic model}
The input-queued switch of interest consists of $n$ input ports and $n$ output ports.
For each pair $(i,j)\in\SetJ:=[n]\times[n]$, packets that needs to be transmitted from the $i$-th input port to the $j$-th output port are stored in a queue indexed by $(i,j)$.
We describe below how the number of packets in a queue (queue length) evolves over time.
Time is slotted by nonnegative integers and the length of queue $\brho\in\SetJ$ at the beginning of the $t$-th time slot is denoted by $\vQ_t(\brho)$.

External packets arrive at each queue according to an exogenous stochastic process.
Let $\vA_t(\brho) \in \IntegerSet_+$ represent the number of arrivals to queue $\brho\in\SetJ$ until time $t$.
Assume that $\{ \vA_{t+1}(\brho)-\vA_{t}(\brho) : t \in \IntegerSet_+, \, \brho\in \SetJ \}$ are independent random variables and that, for fixed $\brho\in\SetJ$,
$\{ \vA_{t+1}(\brho)-\vA_t(\brho) : t \in \IntegerSet_+ \}$ are identically distributed with $\E[\vA_{t+1}(\brho)-\vA_t(\brho)] =: \lambda({\brho})$.
We refer to the $|\SetJ|$-dimensional vector $\blambda\in[0,1]^{\SetJ}$ as the arrival rate vector. Furthermore, $\blambda$ lies in the interior of the {\it capacity region} $\{\blambda\in[0,1]^{\SetJ}, \sum_i \lambda_{ij}<1,  \sum_j \lambda_{ij}<1\}$. 

During each time slot, packets in the queues can be simultaneously transmitted (or departed from the queues) subject to:
\begin{enumerate}
    \item[(1)] At most one packet can be transmitted from an input port;
    \item[(2)] At most one packet can be transmitted to an output port.
\end{enumerate}
Hence, we denote the departure of packets from the queues during a time slot by an $n^2$-dimensional binary vector $\bs=[s(\brho):\brho\in\SetJ]$
such that $s(\brho)=1$ if a packet in queue $\brho$ departs from the queue, and $s(\brho)=0$ otherwise.
We refer to such $\bs$ as a \emph{basic schedule},
and let $\SetI$ denote the set of all basic schedules:
\begin{align}\label{eq:all basic schedules}
\SetI  = \left\{
 \bs\in\{0,1\}^{\SetJ} : \sum_{i\in[n]} s(i,j)\leq 1, \sum_{j\in[n]} s(i,j)\leq 1, \forall i,j\in [n]
\right\}.
\end{align}

Note that the empty basic schedule $\bs$, such that $s(i,j)=0$ for all $(i,j)\in\SetJ$, is indeed a member of $\SetI$.
For $\bs\in\SetI$, let $\vD_t(\bs)$ denote the cumulative number of time slots
devoted to basic schedule $\bs$ until time $t$.
We therefore have
\begin{align}\label{eq:departure process}
    \|\bD_t\|_1=\sum_{\bs\in\SetI}\vD_t(\bs)=t\quad\mbox{and}\quad\|\bD_{t+1}\|_1-\|\bD_t\|_1=1
\end{align}
for every $t\in\IntegerSet_+$.
From the description of arrivals and departures, we can see that $\bQ_t$ evolves according to the following dynamics
\begin{align}\label{eq:dynamics_of_stochastic}
\bQ_{t} \; = \; \bQ_0 + \bA_t - \bD_t\matA,
\end{align}
where $\bQ_0=[\vQ_0(\brho):\brho\in\SetJ]$ is the initial queue lengths and $\matA\in\{0,1\}^{\SetI\times\SetJ}$ is 
the schedule-queue adjacency matrix such that  $A(\bs,\brho)=s(\brho)$ for $\bs\in\SetI$ and $\brho\in\SetJ$. 
We refer to a stochastic process $\{(\bQ_t,\bA_t,\bD_t)\in\IntegerSet_+^{\SetJ}\times\IntegerSet_+^{\SetJ}\times\IntegerSet_+^{\SetI}:t\in\IntegerSet_+\}$
that satisfies \eqref{eq:dynamics_of_stochastic} as a \emph{discrete-time stochastic model for input-queued switches} with the (random) initial state
$\bQ_0\in\IntegerSet_+^{\SetJ}$.

\subsection{Fluid Models}\label{sec:convergence_FL}
%
This section introduces a deterministic process that represents our fluid models for input-queued switches, 
describes the scaled processes of the original stochastic process, and relates them to these fluid models.
The basic set up and ideas can be found in the research literature on fluid limit models, especially the papers of Dai~\cite{dai1995}
and Dai and Prabhakar~\cite{DaiPra00}.
The key concepts concern the tightness and the measures of stochastic processes, which 
leads to the convergence of the subsequences of the scaled processes.  

We introduce a continuous-time deterministic process related to an input-queued switch through the following definition.
\begin{definition}\label{def:fluid_model}
An absolutely continuous deterministic process $\{(\bq_t,\bd_t)\in\RealSet^{\SetJ}\times\RealSet^{\SetI}:t\in\RealSet_+\}$ is called 
a \emph{(input-queued switch) fluid model} with initial state $\bq_0\in\RealSet_+^{\SetJ}$ and arrival rates $\blambda\in [0,1]^{\SetJ}$ if the following conditions hold:
\begin{enumerate}[label=\textbf{(FM\arabic*)}]
    \item\label{fluid:dynamics} $\bq_t= \bq_0 + \blambda t -\bd_t\matA$ for $t\in\RealSet_+$;
    \item\label{fluid:positivity} $\bq_t\geq 0$ for $t\in\RealSet_+$;
    \item\label{fluid:time} $\sum_{\bs\in\SetI} \vd_t(\bs) = t$ (i.e., $\|\bd_t\|_1=t$) and $\bd_t\geq\bzero$ for $t\in\RealSet_+$; 
    \item\label{fluid:increase} For any $\bs\in\SetI$, $\vd_t(\bs)$ is non-decreasing with respect to $t$.
\end{enumerate}
Furthermore, a deterministic process $\{\bmu_t\in\RealSet_+\,:\,t\in\RealSet_+\}$ is called an \emph{(fluid-level) admissible policy} for the input-queued switch
if and only if there exists a fluid model $(\bq_t,\bd_t)$ such that $\bmu_t=\dot{\bd}_t$ for all $t\in\RealSet_+$ at which $\dot{\bd}_t$ exists.
\end{definition}

Note that, since $(\bq_t,\bd_t)$ is absolutely continuous, $\dot{\bq_t}$ and $\dot{\bd_t}$ exist at almost every $t\in\RealSet_+$. 
The following proposition introduces convenient alternative criteria for a fluid-level admissible policy.

\begin{proposition} \label{prop:admissible policy}
Fix $\bq_0\in\RealSet_+^{\SetJ}$ and $\blambda\in[0,1]^{\SetJ}$. 
Let $\{\bmu_t\in\RealSet_+^{\SetI}\,:\,t\in\RealSet_+\}$ be an integrable deterministic process and 
$\{\bq_t\in\RealSet^{\SetJ}\,:\,t\in\RealSet_+\}$ a process satisfying 
$\dot{\bq}_t~=~\blambda -\bmu_{t}\matA$ with initial state  $\bq_0$.
Then, the following statements are equivalent:
\begin{enumerate}[label=\textbf{(AP\arabic*)}]
    \item\label{ap:definition} $\bmu_t$  is a fluid-level admissible policy;
    \item\label{ap:positivity} $\|\bmu_t\|_1=1$ and $\bq_t\geq 0$ for all $t\in\RealSet_+$;
    \item\label{ap:region} $\|\bmu_t\|_1=1$ and $\bmu_t\in\SetU(\bq_t)$ for all $t\in\RealSet_+$, where
    \begin{align}\label{eq:admissible policy region}
        \SetU(\bq):=\left\{\bmu\in [0,1]^{\SetI}\,:\, (\mu A)(\brho)\leq \lambda(\brho) \textrm{ if $q(\brho)=0$} \right\}.
    \end{align}
\end{enumerate}
In this case, $(\bq_t,\bd_t:=\int_0^t \bmu_{t'}dt' )$ is the fluid model associated with the fluid-level admissible policy $\bmu_t$.
\end{proposition}

We next introduce a family of scaled processes, based on the original models indexed by positive integers,
and demonstrate that converging subsequences will have fluid models as their limits, which motivates our fluid optimal control
problems in Section~\ref{sec:fluidopt}.

\subsubsection{Scaled Queueing Processes}
Fix index $r\in\IntegerSet^{+}$ and then let $\{(\bQ^r_t,\bA^r_t,\bD^r_t)\,:\,t\in\IntegerSet_+\}$ be a discrete-time stochastic model with initial state $\bQ^r$
as described in Section~\ref{sec:stochastic model}.
We extend this discrete-time process to a continuous-time process by defining
\begin{align}\label{eq:extention}
    \begin{split}
    \bA^r_t&:=(t-\lfloor t\rfloor)\left(\bA^r_{\lfloor t \rfloor+1}-\bA^r_{\lfloor t\rfloor}\right)+\bA^r_{\lfloor t\rfloor},\\
    \bD^r_t&:=(t-\lfloor t\rfloor)\left(\bD^r_{\lfloor t\rfloor+1}-\bD^r_{\lfloor t\rfloor}\right)+\bD^r_{\lfloor t\rfloor},\\
    \bQ^r_t&:=(t-\lfloor t\rfloor)\left(\bQ^r_{\lfloor t\rfloor+1}-\bQ^r_{\lfloor t\rfloor}\right)+\bQ^r_{\lfloor t\rfloor}\\
    &=\bQ^r+\bA^r_t-\bD^r_t\matA,
    \end{split}
\end{align}
where $\lfloor t\rfloor$ is the largest integer less than or equal to $t$. 

\begin{remark} 
Processes $\vQ^r_t(\brho)$, $\vA^r_t(\brho)$ and $\vD^r_t(\bs)$ are random functions, and 
    every sample path for $(\bQ^r_t,\bA^r_t,\bD^r_t)$ is continuous. 
    We use the notation $\omega^r$ to explicitly denote the dependency on the randomness in the $r$-th system and
the notation $\bomega=[\omega^r\,:\,r\in\IntegerSet^{+}]$ to denote the overall randomness. 
    For example, $\vQ^r_t(\rho;\bomega)=\vQ^r_t(\rho;\omega^r)$ and $\bQ^r_t(\bomega)=\bQ^r_t(\omega^r)$.
\end{remark}
For randomness $\bomega$, the scaled $r$-th system is defined by 
\begin{align}\label{eq:scaled}
\begin{split}
      \lefteqn{(\hat{\bQ}^r_t(\bomega),\;\hat{\bA}^r_t(\bomega),\;\hat{\bD}^r_t(\bomega))}\\
    &\qquad\qquad:=
    \left(r^{-1}\bQ^r_{rt}(\bomega),\;
    r^{-1}\bA^r_{rt}(\bomega),\;
    r^{-1}\bD^r_{rt}(\bomega)
    \right).
\end{split}
\end{align}
We assume that the initial state of the $r$-th system satisfies
\begin{align*}
 r^{-1}\bQ^{r}_0\Rightarrow {\bq}_0, \quad \hbox{as } r\to \infty,  
\end{align*}
for a (deterministic) point $\bq_0\in\RealSet_+^{\SetJ}$, where the convergence is understood to be convergence in distribution.

\subsubsection{Tightness and Convergence}
%
%
%
For a fixed sample path $\bomega$, from~\eqref{eq:departure process} and \eqref{eq:extention}, we have $\hat{\vD}_0(\brho;\bomega)=0$ and $\hat{\vD}_{t}(\brho;\bomega)\leq\|\hat{\bD}_t(\bomega)\|_1=t$ so that
$    \hat{\vD}^{r}_{t}(\brho;\bomega)-\hat{\vD}^{r}_{t'}(\brho;\bomega)~\leq~(t-t')$,
for any $r>0$ and $t\geq t'\geq 0$.
This implies the tightness of the process $\hat{\vD}^{r}_{t}$; see, e.g., \cite{billingsley2013convergence}.

Meanwhile, from the functional strong law of large numbers (see, e.g., \cite{yao2001fundamentals}), we have
\begin{align*}
    \lim_{r\to\infty} \sup_{0\leq t\leq T} | \hat{\vA}_t^{r}(\brho;\bomega) - \lambda(\brho)t|=0
\end{align*}
almost surely.
We therefore have that, almost surely, for each sample path $\bomega$ and any sequence $\{r_k\}$ such that $\lim_{k\to\infty} r_k=\infty$,
there exists a subsequence $\{r_{k_l}\}$ and absolutely continuous deterministic process $(\bq_t,\bd_t)$,
which is a fluid model in Definition~\ref{def:fluid_model},
 such that
\begin{align*}
    (\hat{\bQ}^{r_{k_l}}_t(\bomega),\hat{\bD}^{r_{k_l}}_t(\bomega))\to(\bq_t,\bd_t)
\end{align*}
uniformly on all compact sets as $l\to\infty$.

\begin{remark}
The conditions~\ref{fluid:dynamics} to~\ref{fluid:increase} are necessary conditions for all the fluid limits, and they do not uniquely determine a fluid limit,
even under a fixed admissible scheduling policy.
Such a lack of uniqueness for the fluid limits and its implications for queueing networks are discussed at length in \cite{Bramson1998}.
For certain special cases, with extra conditions on the policies, fluid limits can be shown to be unique;
see, e.g., \cite{shah2012} for input-queued switches.
Our interest, however, is in solving optimal scheduling control problems within the context of the fluid models.
With conditions such as~\ref{fluid:dynamics} and~\ref{fluid:increase}, fluid limit results are generally established for converging subsequences;
similar results can be found in \cite{dai1995} for queueing networks.   
\end{remark}

\subsection{Fluid Model Optimal Control Problems}
\label{sec:fluidopt}
We now formulate the optimal scheduling control problem of interest within the context of the fluid models of input-queue switches.
To this end, we define as follows the total discounted delay cost over the entire time horizon under a fluid-level admissible policy $\{\bmu_t\,:\,t\in\RealSet_+\}$
with initial state $\bq_0$:
\begin{align*}
    c(\bmu_t;\bq_0)~:=~\int_0^\infty e^{-\beta t}\bc\cdot\bq_t dt,
\end{align*}
where $\bq_t$ is the deterministic function defined in \ref{fluid:dynamics} with $\bd_t:=\int_0^t \bmu_s ds$ and initial state $\bq_0$,
$\beta$ is the discount factor, and $\bc\in(\RealSet^+)^{\SetJ}$ is the vector of cost coefficients.
Specifically, we seek to find a fluid-level admissible scheduling policy with the following objective:
\begin{align*}
\textrm{Minimize $c(\bmu_t;\bq_0)$ over all admissible policies $\{\bmu_t\,:\,t\in\RealSet_+\}$}.
\end{align*}
From \ref{ap:positivity} in Proposition~\ref{prop:admissible policy},
this control problem can be formulated as
\begin{align}\label{eq:optimal control problem}
    \begin{split}
        \textrm{minimize} & \qquad \int_0^\infty e^{-\beta t}\bc\cdot\bq_t dt\\
        \textrm{subject to} & \qquad \dot\bq_t=\blambda-\bmu_t\matA,\quad\forall t\in\RealSet_+,\\
        & \qquad \bq_t\geq \bzero,\quad\forall t\in\RealSet_+,\\
        & \qquad \bmu_t\in\SetU,\quad\forall t\in\RealSet_+,
    \end{split}
\end{align}
where $\SetU=\{\bmu\in [0,1]^{\SetI}\,:\,\|\bmu\|_1=1\}$ and the initial state of $\bq_t$ is $\bq_0$.

In the remainder of this section, we exploit results in optimal control theory and derive necessary and sufficient conditions for the optimality
of Problem~\eqref{eq:optimal control problem}.
As previously noted, the Pontryagin Maximum Principle~\cite{PoBoGa+62} typically only provides necessary conditions for optimality, but
these necessary conditions become sufficient under certain conditions that we show to be the case for our optimal control problem.
The Hamiltonian function $H$ and Lagrangian function $L$ corresponding to \eqref{eq:optimal control problem} are respectively defined by
\begin{align*}
    H(\bq,\bmu,\tilde{\bp};t)&:=-e^{-\beta t}\bc\cdot\bq+(\blambda-\bmu\matA)\tilde{\bp},\\
    L(\bq,\bmu,\tilde{\bp},\tilde{\bbeta};t)&:=-e^{-\beta t}\bc\cdot\bq+(\blambda-\bmu\matA)\tilde{\bp}+\bq\cdot\tilde{\bbeta},
\end{align*}
where $\bq$, $\tilde{\bp}$, $\tilde{\bbeta}$ $\in\RealSet^{\SetJ}$ and $\bmu\in\RealSet^{\SetI}$.
We also define
\begin{align*}
   H^*(\bq,\tilde{\bp};t):=\max\left\{H(\bq,\bmu,\tilde{\bp};t)\,:\,\bmu\in\SetU\right\}. 
\end{align*}

Then, from Pontryagin's maximum principle~\cite{PoBoGa+62} under appropriate conditions, we have the following sufficient conditions for an optimal solution
of the optimal control problem.
\begin{lemma}[\protect{\cite[Theorem 8 and 11]{SeSy77}}]\label{lemma:maximum_principle}
    Let $\bq_0$ be the initial condition of a fluid model.  Let $\{\bmu^*_t\in\RealSet_+^{\SetI}:t\in\RealSet_+\}$ be a fluid-level admissible policy,
and let $\bq^*_t=\bq_0+\blambda t + \int_0^t\bmu^*_{t'}  \matA dt'$ be the associated queue length process.
    Assume there exist a process $\{\tilde{\bp}_t\in\RealSet^{\SetJ}:t\in\RealSet_+\}$ with piecewise continuous $\dot{\tilde{\bp}}_t$
    and a process $\{\tilde{\bbeta}_t\in\RealSet^{\SetJ}:t\in\RealSet_+\}$ such that the following conditions are satisfied:
    \begin{enumerate}
        \itemsep0.3em 
        \item[(i)] $H^*(\bq^*_t,\tilde{\bp}_t;t)=H(\bq^*,\bmu^*_t,\tilde{\bp}_t;t)$;
        \item[(ii)] $\dot{\tilde{\bp}}_t=-L'_{\bq}(\bq^*_t,\bmu^*_t,\tilde{\bp}_t,\tilde{\bbeta}_t;t)=-e^{-\beta t}\bc+\tilde{\bbeta}_t$;
        \item[(iii)] $\bq^*_t\cdot\tilde{\bbeta}_t=0$, $\tilde{\bbeta}_t\geq\bzero$;
        \item[(iv)] $\liminf_{t\to\infty} \tilde{\bp}_t\cdot(\bq^*_t-\bq_t)\leq 0$ for any fluid model $(\bq_t,\bd_t)$ with initial condition $\bq_0$;
        \item[(v)] $H^*(\bq, \tilde{\bp}_t;t)$ is concave in $\bq$;
        \item[(vi)] $\bg(\bq):=\bq$ is quasiconcave in $\bq$ and differentiable in $\bq$ at $\bq^*_t$.
    \end{enumerate}
    Then, $\{\bmu^*_t:t\in\RealSet_+\}$ is an optimal solution to  problem~\eqref{eq:optimal control problem}.
\end{lemma}

Observe, however, that by the definition of $H$ and $H^*$, we obtain
\begin{align*}
H^*(\bq, \tilde{\bp}_t;t)&=\max\left\{H(\bq,\bmu,\tilde{\bp};t)\,:\,\bmu\in\SetU\right\}\\
&=
-e^{-\beta t}\bc\cdot\bq+\max\left\{(\blambda-\bmu\matA)\tilde{\bp}\,:\,\bmu\in\SetU\right\},   
\end{align*}
which is linear in $\bq$.
Further observe $\bg(\bq)=\bq$ are linear in $\bq$.
Therefore, conditions (v) and (vi) are satisfied regardless of the choice of $\bq^*_t$, $\bmu^*_t$, $\tilde{\bp}_t$, and $\tilde{\bbeta}_t$.
Hence, we need only check conditions (i)-(iv) to prove the optimality of $\{\bmu^*_t:t\in\RealSet_+\}$.
The following proposition provides an alternative set of sufficient conditions for an optimal solution of the optimal control problem. 

\begin{proposition}\label{prop:max_principle} 
Let $\bq_0$ be the initial condition of a fluid model.  Let $\{\bmu^*_t\in\RealSet_+^{\SetI}:t\in\RealSet_+\}$ be a fluid-level admissible policy,
and let $\bq^*_t=\bq_0-\blambda t + \int_0^t\bmu^*_{t'}  \matA dt'$ be the associated queue length process.
    Assume there exists a continuous process $\{\bp_t\in\RealSet^{\SetJ}:t\in\RealSet_+\}$ with piecewise continuous $\dot{\bp}_t$
    and a process $\{\bbeta_t\in\RealSet_+^{\SetJ}:t\in\RealSet_+\}$ such that the following conditions are satisfied:
    \begin{enumerate}[label=\textbf{(C\arabic*)}]
        \item\label{cond:optimality} $\bmu^*_t\in\arg\max\left\{\mu A p_t\,:\,\bmu\in\SetU\right\}$;
        \item\label{cond:diff} $\dot{\bp}_t-\beta\bp_t=\bc-\bbeta_t$;
        \item\label{cond:slackness} $\bq^*_t\cdot\bbeta_t=0$, $\bq^*_t\geq 0$, $\bbeta_t\geq 0$;
        \item\label{cond:endpoint} $\liminf_{t\to\infty} \bp_t\cdot(\bq^*_t-\bq_t)\geq 0$ for any fluid model $(\bq_t,\bd_t)$ with initial condition $\bq_0$.
    \end{enumerate}
    Then, $\{\bmu^*_t:t\in\RealSet_+\}$ is an optimal solution to the optimal control problem~\eqref{eq:optimal control problem}.
\end{proposition}

%% file: optimal.tex
\section{Optimal Control}
\label{sec:control}
In this section, we present and analyze algorithms that render the optimal fluid-cost scheduling policy,
namely the optimal solution to the control problem~\eqref{eq:optimal control problem} of Section~\ref{sec:fluidopt}.
We first provide and recall some technical preliminaries, including additional notation.
Then we present a critical threshold result for a family of linear programs, followed by the optimal control algorithm that exploits
a critical threshold at each state of the system.

\subsection{Technical Preliminaries} 
\label{sub:terminologies}
%
We refer to the stochastic model in Section~\ref{sec:stochastic model} as the pre-limit model and refer to the fluid model in Section~\ref{sec:convergence_FL}
as the limit system.
For the pre-limit model, recall that a basic schedule is a collection of queues from each of which a packet can depart simultaneously, where
$\SetJ:=[n]\times[n]$ denotes the set of queues.
A basic schedule is represented by a $|\SetJ|$-dimensional binary vector $\bs=[s(\brho)\in\{0,1\}:\brho\in{\SetJ}]$, where $s(\brho)=1$ if and only if $\brho$
is in the collection composing the basic schedule.
For $\brho\in\SetJ$ and $\bs\in\SetI$, we use $\brho\in\bs$ if $s(\brho)=1$.
For a basic schedule $\bs\in\SetI$, with $\SetI$ the set of all basic schedules given in~\eqref{eq:all basic schedules}, we define the \emph{weight} of $\bs$ by 
\begin{align*}
    w(\bs)~:=~\sum_{\brho\in\bs} c(\brho), 
\end{align*}
where $\bc\in(\RealSet^+)^{\SetJ}$ is the cost coefficient vector introduced in~\eqref{eq:optimal control problem}.


While time in the pre-limit system is discrete with queue-length vector $\bQ_t\in\IntegerSet_+^{\SetJ}$ at time $t\in\IntegerSet_+$, 
time in the limit system is continuous with the state space of (fluid) queue-length vectors $\bq_t$ given by $\RealSet_+^{\SetJ}$.
From Proposition~\ref{prop:admissible policy}, we define a \emph{(fluid-level) schedule} by a convex combination of basic schedules and
represent it as an $|\SetI|$-dimensional vector $\bmu=[\mu(\bs)\in[0,1]:\bs\in\SetI]$ with $\|\bmu\|_1=1$,
where $\mu(\bs)$ is the coefficient of schedule $\bs$.
Furthermore, schedule $\bmu$ is \emph{admissible} at state $\bq\in\RealSet_+^{\SetJ}$ if and only if $\bmu\in\SetU(\bq)$,
as defined in~\eqref{eq:admissible policy region}.

\subsection{Critical Thresholds}
We now introduce, for each state $\bq\in\RealSet_+^{\SetJ}$, a family of linear programming problems, indexed by non-negative real numbers,
from which we construct an (admissible) schedule associated with the linear program.
These schedules are instrumental to the development of the optimal control algorithms in Section~\ref{sec:algorithms}. 
For a given state $\bq$ and a real value $\tau\in\RealSet_+$, define sets $\SetI_{\tau}\subset\SetI$ and $\SetJ_{\bq}\subset\SetJ$ by $\SetI_{\tau}:=\left\{\bs\in\SetI\;:\;w(\bs)\geq \tau\right\}$, $\SetJ_{\bq}:=\{\brho\in\SetJ\;:\;q(\brho)=0\}$,
respectively, and define an $|\SetI_{\tau}|$-dimensional vector 
\begin{align*}
  \bw_{\tau}:=[w(\bs)-\tau:\bs\in\SetI_{\tau}]\in\RealSet_+^{\SetI_{\tau}}.
\end{align*}
Then, for $\tau$ with $\SetI_{\tau}\neq\emptyset$, we formulate the following linear programming problem:
\begin{align}\label{eq:primal}\tag{$P_{\bq,\tau}$}
    \begin{split}
            \textrm{max}  \quad \bw_{\tau}\cdot \bnu, \quad
            \textrm{s.t.} & \quad \bnu \matA_{{\tau},{\bq}} \leq \blambda_{{\bq}},
             \quad \bnu \geq \bzero,
    \end{split}
\end{align}
where
\begin{align*}
    \matA_{{\tau},{\bq}}~&:=~ 
    [ A(\bs,\brho)\,:\,\bs\in\SetI_{\tau},\,\brho\in\SetJ_{\bq}]~\in~\{0,1\}^{\SetI_{\tau}\times\SetJ_{\bq}},\\
    \blambda_{{\bq}}~&:=~ 
    [\lambda(\brho)\,:\,\brho\in\SetJ_{\bq}]\,\in\,[0,1]^{\SetJ_{\bq}},
\end{align*}
and $\bnu\in\RealSet^{\SetI_{\tau}}$ is the vector of decision variables.
Note that, if $\tau=0$, then $\SetI_{0}=\SetI$ and $\bw_0=\matA \bc$.

\begin{remark}
The feasible region for Problem~\eqref{eq:primal} is nonempty because $\bnu=\bzero$ obviously satisfies all constraints.
From any feasible vector $\bnu$ for Problem~\eqref{eq:primal}, if we define $\bmu\in\RealSet^{\SetI}$ by 
\begin{align*}
    \mu(\bs)~=~
    \begin{cases}
        \nu(\bs)  & \textrm{if $\bs\in\SetI_{\tau}$}\\
        0 & \textrm{otherwise}
    \end{cases},
\end{align*}
then we have $\bmu\in\SetU(\bq)$ due to the constraints in Problem~\eqref{eq:primal}.
Thus, when $\|\bmu\|_1=\|\bnu\|_1=1$, $\bmu$ is an admissible schedule at state $\bq$.
\end{remark}    

The next theorem shows the existence of a specific $\tau\in\RealSet_+$ for each state $\bq$,
from which we can construct an admissible schedule associated with an optimal solution to Problem~\eqref{eq:primal}.
\begin{theorem}\label{thm:critical threshold}
    For any state $\bq$, there exists a $\tau=\tau(\bq)\in\RealSet_+$ such that Problem~\eqref{eq:primal} has an optimal solution $\bnu$ that can be extended to an admissible schedule at state $\bq$; namely, $\|\bnu\|_1=1$.
We call such $\tau$ a \emph{critical threshold} of state $\bq$.
\end{theorem}

In the remainder of this section, we provide the basic arguments for establishing Theorem~\ref{thm:critical threshold} by devising a search algorithm
for critical thresholds that will terminate in a finite number of iterations.

First, letting $\gamma$ denote the optimal value of Problem~\eqref{eq:primal}, it is obvious that $\tau$ is a critical threshold at state $\bq$
if and only if the following set is nonempty:
\begin{align}\label{eq:check critical threshold}
  \SetQ(\bq,\tau,\gamma)
    :=\left\{\bnu\geq 0 \in\SetI_{\tau}\,:\,  \bw_{\tau}\cdot\bnu=\gamma, \|\bnu\|_1=1,\  \bnu \matA_{\tau,{\bq}} \leq \blambda_{\bq}\right\}.
\end{align}
Note that all constraints in~\eqref{eq:check critical threshold} are linear and $\SetQ({\bq,\tau,\gamma})$ is a polyhedron, which implies that the
emptiness of the set $\SetQ(\bq,\tau,\gamma)$ can be checked quickly through the solution of a linear program. 

Define $\SetW:=\{w(\bs)\,:\,\bs\in\SetI\}=\{\tau_1,\tau_2,\dots\}$ to be the ordered set of all (distinct) weights of schedules in $\SetJ$ with $\tau_i>\tau_{i+1}$
for $i=1,2,\dots$.
Algorithm~\ref{alg:critical threshold in W} then checks if $\SetW$ contains a critical threshold and finds one if it exists. 

{\small
\begin{algorithm}
\caption{Algorithm to find a critical threshold at state $\bq$ in $\SetW$} \label{alg:critical threshold in W}
\textbf{Input:} None, \qquad
\textbf{Output:} An integer 
\begin{algorithmic}[1]
\State\label{alg:check critical threshold:def of h}
{Set $l=1$ and 
\begin{align*}
 h=\min\{k:\exists \bs\in\SetJ \textrm{ such that } w(\bs)=\tau_k,\ q(\brho)\neq 0\ \forall\brho\in\bs \} 
\end{align*}
}
\State{Solve Problem~\eqref{eq:primal} with $\tau=\tau_l$, obtain an optimal value $\gamma_l$ and an optimal solution $\bnu^*$}\label{alg:check critical threshold:initial check start}
\If{$\SetQ(\bq,\tau_l,\gamma_l)\neq\emptyset$}
    \State\Return{$l$}
\EndIf
\State{Solve Problem~\eqref{eq:primal} with $\tau=\tau_h$, obtain an optimal value $\gamma_h$ and an optimal solution $\bnu^*$}
\If{$\SetQ(\bq,\tau_h,\gamma_h)\neq\emptyset$}
    \State\Return{$h$}
\EndIf\label{alg:check critical threshold:initial check end}
\While{$l<h-1$}
    \State{Set $m=\lfloor\frac{l+h}{2}\rfloor$ and $\tau=\tau_{m}$}
    \State{Solve Problem~\eqref{eq:primal} with $\tau=\tau_m$, obtain an optimal value $\gamma_m$ and an optimal solution $\bnu^*$}
    \If{$\SetQ(\bq,\tau_m,\gamma_m)\neq\emptyset$}
    \State\Return{$m$}\label{alg:check critical threshold:return m}
    \Else
    \If{$\|\bnu^*\|_1>1$}\label{alg:check critical threshold:start update l h}
    \State{Set $h=m$}
    \Else
    \State{Set $l=m$}
    \EndIf\label{alg:check critical threshold:end update l h}
    \EndIf
\EndWhile
\State\Return{$-l$}\label{alg:check critical threshold:return -l}
\end{algorithmic}
\end{algorithm}}


The next proposition shows that, if the algorithm returns a positive integer $m$, then $\tau_m$ is a critical threshold of state $\bq$.


\begin{proposition}\label{prop:critical threshold in W}
    If there exists a critical threshold in $\SetW$, Algorithm~\ref{alg:critical threshold in W} returns a positive integer $m$ such that $\tau_m\in\SetW$ is a critical threshold.
    Otherwise, it returns $-l$ (where $l\in\IntegerSet^{+}$) such that\\
\indent $1$-norm of any optimal solution to~\eqref{eq:primal} with $\tau=\tau_l$ is $<1$;\\
\indent $1$-norm of any optimal solution to~\eqref{eq:primal} with $\tau=\tau_{l+1}$ is $>1$.
\end{proposition}

\begin{remark}
    Algorithm~\ref{alg:critical threshold in W} has $O(\log|\SetW|)$ iterations because $(h-l)$ is almost one greater than half of the previous value of $(h-l)$ in the algorithm.
\end{remark}

When Algorithm~\ref{alg:critical threshold in W} returns a critical threshold $\tau_m$ of state $\bq$, for positive integer $m$, we have the key element needed
for our optimal control policy in this case, as we will see in Algorithm~\ref{alg:optimal_control}.
Otherwise, we exploit the results from Algorithm~\ref{alg:critical threshold in W} to obtain the desired critical threshold for state $\bq$.
Henceforth, assume that $\SetW$ does not contain any critical threshold.
From the above results, in this case, Algorithm~\ref{alg:critical threshold in W} returns $-l$ for some $l\in\IntegerSet^{+}$;
and if a critical threshold exists in $\RealSet_+$ (but not in $\SetW$), then it is between $\tau_{l+1}$ and $\tau_{l}$. 
We define $\bar{\bw}:=[w(\bs):\bs\in\SetI_{\tau_l}]$ and formulate another linear optimization problem for $\tau\in(\tau_{l+1},\tau_l)$:
\begin{align}\label{eq:primal sub}\tag{$P'_{\bq,\tau}$}
    \begin{split}
            \textrm{max} & \quad \bar{\bw}\cdot \bnu-\tau\|\bnu\|_1, \quad
            \textrm{s.t.}  \quad \bnu \matA_{\tau_l,\bq} \leq \blambda_{\bq}, 
             \quad \bnu \geq \bzero,
    \end{split}
\end{align}
where $\bnu\in\RealSet^{\SetI_{\tau_l}}$ is a vector of decision variables.

The following proposition then allows us to find a critical threshold of state $\bq$ in $(\tau_{l+1},\tau_{l})$ based on the solution to the linear program~\eqref{eq:primal sub}.
\begin{proposition}\label{prop:critical threshold not in W 1}
    Assume that $\SetW$ does not contain any critical threshold and let $-l$ be the output of Algorithm~\ref{alg:critical threshold in W}
for some positive integer $l\in\IntegerSet^{+}$.  Then,
    \begin{enumerate}
        \item[(i)] For $\tau\in(\tau_{l+1},\tau_{l})$, Problem~\eqref{eq:primal sub} is equivalent to Problem~\eqref{eq:primal};
        \item[(ii)] The feasible region of Problem~\eqref{eq:primal sub} is a polytope (bounded polyhedron);
        \item[(iii)] All optimal solutions to Problem~\eqref{eq:primal sub} with $\tau=\tau_{l+1}$ have $1$-norm greater than $1$.
    \end{enumerate}
\end{proposition}

\begin{remark} 
Note that in Problem~\eqref{eq:primal sub}, only the objective function depends on $\tau$ and feasible sets do not depend on $\tau$.
Since Problem~\eqref{eq:primal} is equivalent to Problem~\eqref{eq:primal sub} for $\tau\in(\tau_{l+1},\tau_{l})$, 
we can verify if $\tau$ is a critical threshold by checking the emptiness of the set
\begin{align}\label{eq:check critical threshold sub}
\begin{split}
    \lefteqn{\SetQ'(\bq,\tau,\gamma)}\\
    &:=\left\{\bnu\in\SetI_{\tau_l}:\bar{\bw}\cdot\bnu-\tau=\gamma, \|\bnu\|_1=1,\bnu \matA_{\tau_l,\bq} \leq \blambda_{\bq},\bnu \geq \bzero\right\},
\end{split}
\end{align}
where $\gamma$ is the optimal value of Problem~\eqref{eq:primal sub}. 
\end{remark}

Now, we present an algorithm that obtains a critical threshold of state $\bq$ in $(\tau_{l+1},\tau_{l})$.
{\small
\begin{algorithm}
\caption{Algorithm to find a critical threshold at state $\bq$ in $(\tau_{l+1},\tau_{l})$} \label{alg:critical threshold not in W}
\textbf{Input:} integer $l$ such that\\
\hspace*{0.5in} $1$-norm of any optimal solution to Problem \eqref{eq:primal sub} with $\tau=\tau_l$ is less than $1$ \\
\hspace*{0.5in} $1$-norm of any optimal solution to Problem \eqref{eq:primal sub} with $\tau=\tau_{l+1}$ is greater than $1$ \\
\textbf{Output:} a critical threshold $\tau\in(\tau_{l+1},\tau_{l})$
\begin{algorithmic}[1]
\State{Set $\bar{\bw}=[w(\bs):\bs\in\SetI_{\tau_l}]$, and $k=0$}
\State{Set $\tau^L_0=\tau_l$ and obtain a basic optimal solution $\bnu^L_0$ to Problem~\eqref{eq:primal sub} with $\tau=\tau^L_0$}
\State{Set $\tau^S_0=\tau_{l+1}$ and obtain a basic optimal solution $\bnu^S_0$ to Problem~\eqref{eq:primal sub} with $\tau=\tau^S_0$}
\While{True}
\State{Set $$\tau^M_k:=\frac{\bar{\bw}\cdot(\bnu^S_k-\bnu^L_k)}{\|\bnu^S_k\|_1-\|\bnu^L_k\|_1}$$}
\State{Solve Problem \eqref{eq:primal sub} with $\tau=\tau^M_k$, obtain optimal value $\gamma^*$ and basic optimal solution $\bnu^M_k$}
\If{$\SetQ'(\bq,\tau^M_k,\gamma^*)\neq\emptyset$}\label{line:condition}
\State\Return{$\tau^M_k$}
\Else
\If{$\|\bnu^M_k\|_1>1$}\label{alg:critical threshold not in W:start update l h}
\State{Set $(\tau^S_{k+1},\bnu^S_{k+1})=(\tau^M_k,\bnu^M_k)$}
\State{\quad and $(\tau^L_{k+1},\bnu^L_{k+1})=(\tau^L_k,\bnu^L_k)$}
\Else
\State{Set $(\tau^L_{k+1},\bnu^L_{k+1})=(\tau^M_k,\bnu^M_k)$}
\State{\quad and $(\tau^S_{k+1},\bnu^S_{k+1})=(\tau^S_k,\bnu^S_k)$}
\EndIf \label{alg:critical threshold not in W:end update l h}
\State{Set $k=k+1$}
\EndIf
\EndWhile
\end{algorithmic}
\end{algorithm}}




The next proposition establishes that this algorithm provides a critical threshold of state $\bq$.

\begin{proposition}\label{prop:critical threshold not in W}
    Assume that $\SetW$ does not contain any critical threshold and $-l$ is the output of Algorithm~\ref{alg:critical threshold in W} for some positive integer $l\in\IntegerSet^{+}$. 
    Then, Algorithm~\ref{alg:critical threshold not in W} with input $l$ returns a critical threshold in a finite amount of time.
\end{proposition}

To summarize, the following algorithm combines Algorithm~\ref{alg:critical threshold in W} and Algorithm~\ref{alg:critical threshold not in W}
to produce a critical threshold for any state $\bq$.
{\small
\begin{algorithm}
\caption{Algorithm to find a critical threshold at state $\bq$}\label{alg:critical threshold}
\textbf{Input:} State $\bq$
\textbf{Output:} a critical threshold $\tau=\tau(\bq)$
\begin{algorithmic}[1]
\State{Set $m$ be the output of Algorithm~\ref{alg:critical threshold in W} with input $\bq$}
\If{$m>0$}
\State\Return{$\tau_m$}
\Else
\State\Return{the output of Algorithm~\ref{alg:critical threshold not in W} with input $l=-m$}
\EndIf
\end{algorithmic}
\end{algorithm}}

\subsection{Optimal Control Algorithm}
\label{sec:algorithms}
By exploiting the critical threshold for any state $\bq$ from the previous section, we now introduce an optimal control algorithm and
show that it renders an optimal solution to the optimal control problem~\eqref{eq:optimal control problem}.
{\small
\begin{algorithm}
\caption{Optimal Control Algorithm for initial state $\bq_{t=0}$}\label{alg:optimal_control}
\begin{algorithmic}[1]
\State{Set $k=0$, $t_0=0$, and $\bq^*_{0}=\bq_{t=0}$}
\While{$t_k<\infty$}
    \State{Let $\tau_k$ be the output of Algorithm~\ref{alg:critical threshold} with input $\bq=\bq^*_{t_k}$}
    \State{Let $\gamma_k$ be the optimal value of  Problem~\eqref{eq:primal} with $\bq=\bq^*_{t_k}$ and $\tau=\tau_k$}
    \State{Find a point $\bnu_k\in\SetQ(\bq^*_{t_k},\tau_k,\gamma_k)$ in \eqref{eq:check critical threshold}}\label{alg:line:find bnu_k}
    \State{Define $\bmu^*\in\RealSet^{\SetI}$ by 
    \begin{align*}
        \mu^*(\bs)=\begin{cases}
        \nu_k(\bs) & \textrm{if $\bs\in\SetI_{\tau_k}$}\\
        0 & \textrm{otherwise}
        \end{cases}
    \end{align*}}\label{alg:line:bmu}
    \State{Set
    \begin{align*}
      \lefteqn{t_{k+1}=t_k}\\
      &+\min\left\{ \frac{q_{t_{k}}(\brho)}{(\mu^*\,A)(\brho)-\lambda(\brho)}\,:\,\brho\in\SetJ\backslash\SetJ_{\bq_{t_k}^*},\  (\mu^*\,A)(\brho)-\lambda(\brho)>0\right\}
    \end{align*}
    }\label{alg:definition of t_{k+1}}
    \State{Set $\bmu^*(t)=\bmu^*$ for $t\in[t_k,t_{k+1})$ and $\bq^*_{t}=\bq^*_{t_k}+(t-t_k)\blambda-(t-t_k)\bmu^* \matA$ for $t\in[t_k,t_{k+1}]$}\label{alg:line:definition of q_t}
    \State{Set $k=k+1$}
\EndWhile
\end{algorithmic}
\end{algorithm}}

The next proposition shows that the above algorithm produces a fluid-level admissible policy.
\begin{proposition} \label{prop:well-definedness of t_{k+1}}
    In Algorithm~\ref{alg:optimal_control}, we have that $\bmu^*_t$ is a fluid-level admissible policy and $\bq^*_t$ is the continuous process satisfying
$\dot{\bq^*_t}=\blambda-\bmu^*_t\,\matA$ with initial state $\bq_{t=0}$.
\end{proposition}

Now, we prove the stability of the system under the scheduling policy $\bmu^*_t$ in Algorithm~\ref{alg:optimal_control}.
\begin{theorem}
	\label{thm:throughput_optimal}
	Assume that the arrival rate vector $\blambda$ is inside the capacity region. Then, the schedule produced by Algorithm~\ref{alg:optimal_control} empties the system in finite time. 
    Moreover, if $\bq^*_{T}=\bzero$ for some $T\geq 0$, then $\bq^*_t=\bzero$ for all $t\geq T$.
\end{theorem}
The second result in the above theorem claims that Algorithm~\ref{alg:optimal_control} is weakly stable, the definition of which is as follows.
\begin{definition}[\protect{\cite[Definition 6]{DaiPra00}}]\label{def:weak stability}
A fluid-level admissible policy $\bmu_t$ is \emph{weakly stable} if the corresponding fluid queue length process $\{\bq_t\,:\,t\in\RealSet_+\}$ with initial state $\bq_0=\bzero$ satisfies $\bq_t=\bzero$ for all $t\geq 0$.
\end{definition}


We next establish that, under this implication, Algorithm~\ref{alg:optimal_control} is an optimal policy that satisfies Proposition~\ref{prop:max_principle}.
\begin{theorem}\label{thm:optimal_result}
  Assume that the arrival rate vector $\blambda$ is in the capacity region. 
  Then, $(\bq^*_t, \bmu^*_t)$ is an optimal solution to problem~\eqref{eq:optimal control problem}.
\end{theorem}

\subsection{Relationship with c$\bm{\mu}$ Policy}
Given an arrival rate vector $\blambda$ and initial queue length $\bq_0$ such that $\lambda(i,j)=q_0(i,j)=0$ for all $i\in[n]$ and $j\in[n]\setminus\{1\}$,
the $n\times n$ input-queued switch is equivalent to $n$ parallel queues with one server. 
The c$\mu$-policy is well-known for this case to be an optimal policy that minimizes the discounted total cost over an infinite horizon in both the stochastic and fluid models (see~\cite{Cox61} and~\cite{Bauerle2000});
and, in this case, Algorithm~\ref{alg:optimal_control} follows the c$\mu$-policy in the fluid model.

However, the c$\mu$-policy is not optimal for the $n\times n$ input-queued switch in general.
Consider a $3\times 3$ input-queued switch fluid model such that $\lambda(i,j)=0.45$ if $ (i,j)=(1,1),(1,2),(2,1),(2,3)$, and zero otherwise; $c(i,j)=1$ if  $(i,j)=(1,2),(2,3)$, $c(i,j)=0.5$ if  $(i,j)=(2,1)$, $c(i,j)=0.1$ if  $(i,j)=(1,1),(2,3)$,  and zero otherwise;
$\bq_0=\bzero$. 
Then, according to the c$\mu$-policy, the admissible schedule at $\bq$ with $q(1,2)=q(2,3)=q(2,1)=0$ becomes
\begin{align*}
  \mu(\bs)~=~
  \begin{cases}
    0.45 & \textrm{for $\bs$ such that $s(1,2)=s(2,3)=1$}\\
    0.45 & \textrm{for $\bs$ such that $s(2,1)=1$}\\
    0.10 & \textrm{for $\bs$ such that $s(1,1)=1$}\\
    0 & \textrm{otherwise}
  \end{cases}.
\end{align*}
Hence, the queue lengths for $(1,2)$, $(2,3)$ and $(2,1)$ are maintained at zero but the queue length for $(1,1)$ increases with rate $0.45-0.10=0.35$,
which shows that the c$\mu$-policy is not weakly stable.

On the other hand, according to Theorem~\ref{thm:throughput_optimal}, Algorithm~\ref{alg:optimal_control} is weakly stable.
In this example, the critical threshold at $\bq_0=\bzero$ is $\tau=0$ and the admissible schedule is
\begin{align*}
  \mu^*(\bs)~=~
  \begin{cases}
    0.45 & \textrm{for $\bs$ such that $s(1,2)=s(2,1)=1$} \\
    0.45 & \textrm{for $\bs$ such that $s(1,1)=s(2,3)=1$}\\
    0 & \textrm{otherwise}
  \end{cases},
\end{align*}
which maintains the system to be empty.


%% file: proofs.tex
\section{Proofs of Main Results}\label{sec:proofs}
In this section, we turn to consider the proofs of our main results.

\subsection{Proof of Proposition~\ref{prop:admissible policy}}
From the differential equation and the initial state of $\bq_t$, we have 
\begin{align} \label{eq:integral for q_t}
    \bq_t~=\bq_0+\blambda t-\int_0^t \bmu_{t'}\matA dt'~=~~\bq_0+\blambda t-\left(\int_0^t \bmu_{t'} dt'\right)\matA. 
\end{align}
Therefore, $\{\bq_t:t\in\RealSet_+\}$ is well-defined and differentiable everywhere.
Now, we show that \ref{ap:region} $\Rightarrow$ \ref{ap:positivity}  $\Rightarrow$  \ref{ap:definition}  $\Rightarrow$  \ref{ap:region}.

Assume that $\bmu_t$ satisfies $\|\bmu_t\|_1=1$ and $\bmu_t\in\SetU(\bq_t)$ for all $t\in\RealSet_+$.
We claim that $\bq_t\geq \bzero$ for all $t\in\RealSet_+$.
If this is not true, i.e., $q_{t'}(\brho)<0$ for some $\brho\in\SetJ$ at some time $t'$,
then let $t''=\sup\{t<t'\,:\,q_t(\brho)=0\}$ which is well-defined because $q_t(\brho)$ is continuous and $q_0(\brho)=\bq(\brho)\geq 0$.
By the continuity of $q_t(\brho)$, we have that $q_{t''}(\brho)=0$ and  $q_t(\brho)<0$ for all $t\in(t'',t')$.
Hence, $\dot{q}_{t''}(\brho)<0$, which contradicts the fact that $\lambda_{t''}(\brho)\leq (\mu_{t''}A)(\brho)$, and thus
$\bq_t\geq 0$  for all $t\in\RealSet_+$, which proves that \ref{ap:region} implies \ref{ap:positivity}.

Suppose $\|\bmu_t\|_1=1$ and $\bq_t\geq 0$ for $t\in\RealSet_+$. We show that $(\bq_t,\bd_t)$ is a fluid model with $\bd_t:=\int_0^t \bmu_{t'} dt'$.
Conditions~\ref{fluid:dynamics} and \ref{fluid:positivity} immediately follow from \eqref{eq:integral for q_t} and the assumption in \ref{ap:positivity}, respectively. 
Further note that
\begin{align*}
    \|\bd_t\|_1
    =\sum_{\bs\in\SetI} \int_0^t {\bmu}_{t'}(\bs)dt'=\int_0^t \sum_{\bs\in\SetI}\mu_{t'}(\bs)dt'=
    \int_0^t \|\bmu\|_1=t,
\end{align*}
which implies the condition~\ref{fluid:time}. 
Since $\dot{\bd}_t=\bmu_t\geq 0$ for all $t\in\RealSet_+$, the condition~\ref{fluid:increase} also holds, and therefore \ref{ap:positivity} implies \ref{ap:definition}.

Lastly, assume that $\{\bmu_t\,:\,t\in\RealSet_+\}$ is a fluid-level admissible policy and let $(\bq_t,\bd_t)$ be a fluid model with $\dot{\bd}_t=\bmu_t$, which implies $\bd_t=\int_0^t \bmu_{t'} dt'$. From conditions~\ref{fluid:time} and \ref{fluid:increase}, we have 
\begin{align*}
    \|\bmu_t\|_1=\|\dot{\bd}_t\|_1=\sum_{\bs\in\SetI} \dot{\vd}_t(\bs)=\frac{d}{dt}{\left(\sum_{\bs\in\SetI} \vd_t(\bs)\right)}=\frac{d}{dt}\|\bd_t\|_1=1.
\end{align*}
Moreover, from the condition~\ref{fluid:dynamics},
$\bq_t$ is the process such that $\bq_t=\bq_0+\blambda t- \int_0^t\bmu_{t'}\matA dt'$.
If $q_t(\brho)=0$ but $\lambda(\brho)<\mu_t(\brho)$ for some $t\in\RealSet_+$ and $\brho\in\SetJ$, then $\dot{q}_t(\brho)<0$.
Therefore, we have $q_{t'}(\brho)<0$ for $t'\in[t,t+\varepsilon]$ and some $\varepsilon>0$, which contradicts the condition~\ref{fluid:positivity}. 
Hence, we obtain $\bmu_t\in\SetU(\bq_t)$ for $t\in\RealSet_+$, and thus \ref{ap:definition} is a sufficient condition for \ref{ap:region}.

\subsection{Proof of Proposition~\ref{prop:max_principle}}
Define $\tilde{\bp}_t:=-e^{-\beta t}\bp_t$ and $\tilde{\bbeta}_t:=e^{-\beta t}\bbeta_t$.
We then prove that $\tilde{\bp}_t$ and $\tilde{\bbeta}_t$ satisfy the conditions in Lemma~\ref{lemma:maximum_principle}.

From \ref{cond:optimality}, we have
\begin{align*}
    H^*(\bq^*_t,\tilde{\bp}_t;t)&~=~\max\left\{H(\bq^*_t,\bmu,\tilde{\bp}_t;t)\,:\,\bmu\in\SetU\right\}\\
    &~=~\max\left\{-e^{-\beta t}\bc\cdot\bq^*_t+\left(\blambda-\bmu\matA\right)\tilde{\bp}_t\,:\,\bmu\in\SetU\right\}\\
    &~=~-e^{-\beta t}\bc\cdot\bq^*_t+\blambda\cdot\tilde{\bp}_t+e^{-\beta t}\;\max\left\{\bmu\matA\bp_t\,:\,\bmu\in\SetU\right\}\\
    &~=~-e^{-\beta t}\bc\cdot\bq^*_t+\blambda\cdot\tilde{\bp}_t+e^{-\beta t}\bmu^*_t\matA\bp_t\\
    &~=~-e^{-\beta t}\bc\cdot\bq^*_t+\left(\blambda-\bmu^*_t\matA\right)\tilde{\bp}_t\\
    &~=~H(\bq^*_t,\bmu^*_t,\tilde{\bp}_t;t),
\end{align*}
which implies condition (i) of Lemma~\ref{lemma:maximum_principle}.
Condition~\ref{cond:diff} implies
\begin{align*}
    \dot{\tilde{\bp}}&=-e^{-\beta t}\dot{\bp}_t+\beta e^{-\beta t}\bp_t=-e^{-\beta t}\left(\dot{\bp}_t-\beta\bp_t\right)\\
    &=-e^{-\beta t}\left(\bc-\bbeta_t\right)=-e^{-\beta t}+\tilde{\bbeta}_t,
\end{align*}
which proves condition (ii) of Lemma~\ref{lemma:maximum_principle}.
  
Since $\bbeta_t$ is a positive multiple of $\tilde{\bbeta}_t$ and $\bp_t$ is a negative multiple of $\tilde{\bp}_t$,
conditions (iii) and (iv) of Lemma~\ref{lemma:maximum_principle} then follow from conditions~\ref{cond:slackness} and~\ref{cond:endpoint}, respectively.


\subsection{Proof of Proposition~\ref{prop:critical threshold in W}}

We first introduce a key lemma that relates the norms of optimal solutions to Problem~\eqref{eq:primal} with different $\tau$.

\begin{lemma}\label{lemma:monotonicity}
Fix $\tau',\tau''\in\RealSet_+$ with $\tau'>\tau''$.
Let $\bnu'\in\RealSet_+^{\SetI_{\tau'}}$ and $\bnu''\in\RealSet_+^{\SetI_{\tau''}}$ be solutions to  Problem~\eqref{eq:primal} with $\tau=\tau'$ and $\tau=\tau''$, respectively. Then, we have $\|\bnu'\|_1 \leq \|\bnu''\|_1$.
\end{lemma}

\begin{proof}
Note that $\SetI_{\tau'}\subset\SetI_{\tau''}$.
We denote $\bnu''_1\in\RealSet_+^{\SetI_{\tau'}}$ and $\bnu''_2\in\RealSet_+^{\SetI_{\tau''}\backslash\SetI_{\tau'}}$ as the projections of $\nu''(\bs)$
to $\RealSet_+^{\SetI_{\tau'}}$ and $\RealSet_+^{\SetI_{\tau''}\backslash\SetI_{\tau'}}$, respectively;
i.e., $\nu''_1(\bs)=\nu''(\bs)$ for all $\bs\in\SetI_{\tau'}$ and $\nu''_2(\bs)=\nu''(\bs)$ for all $\bs\in\SetI_{\tau''}\backslash\SetI_{\tau'}$, respectively.
Naturally, we have 
\begin{align*}
    \blambda_{{\bq}}~\geq~\bnu''\matA_{\tau'',\bq}
    ~\geq~ \bnu''_1\matA_{\tau',\bq},
\end{align*}
which implies that $\bnu''_1$ is a feasible solution of Problem~\eqref{eq:primal} with $\tau=\tau'$.
Hence, we obtain
\begin{align}\label{eq:lemma_monotonicity}
    \bw_{\tau'}\cdot\bnu''_1~\leq~\bw_{\tau'}\cdot\bnu'
\end{align}
due to the fact that $\bnu'$ is an optimal solution to Problem~\eqref{eq:primal}. On the other hand, we have
\begin{align}\label{eq:lemma_monotonicity:2}
    \begin{split}
    \lefteqn{\bw_{\tau''}\cdot\bnu''}\\
    =&\sum_{\bs\in\SetI_{\tau''}}\left(w(\bs)-\tau''\right)\nu''(\bs)\\
    =&\sum_{\bs\in\SetI_{\tau'}}\left(w(\bs)-\tau''\right)\nu''(\bs)
    +\sum_{\bs\in\SetI_{\tau''}\backslash\SetI_{\tau'}}\left(w(\bs)-\tau''\right)\nu''(\bs)\\
    =&\sum_{\bs\in\SetI_{\tau'}}\left(w(\bs)-\tau'\right)\nu''_1(\bs)+(\tau'-\tau'')\sum_{\bs\in\SetI_{\tau'}}\nu''(\bs)\\
    &+\sum_{\bs\in\SetI_{\tau''}\backslash\SetI_{\tau'}}\left(w(\bs)-\tau''\right)\nu''(\bs)\\
    \leq&\sum_{\bs\in\SetI_{\tau'}}\left(w(\bs)-\tau'\right)\nu''_1(\bs)+(\tau'-\tau'')\sum_{\bs\in\SetI_{\tau'}}\nu''(\bs)\\
    &+\sum_{\bs\in\SetI_{\tau''}\backslash\SetI_{\tau'}}\left(\tau'-\tau''\right)\nu''(\bs)\\
    &=\bw_{\tau'}\cdot\bnu''_1+(\tau'-\tau'')\|\bnu''\|_1,
    \end{split}
\end{align}
where the inequality follows from $w(\bs)<\tau_1$ for all $\bs\in\SetI_{\tau''}\backslash\SetI_{\tau'}$.
Now, if we extend $\bnu'$ to $\tilde{\bnu'}\in\RealSet_+^{\SetI_{\tau''}}$ by
\begin{align*}
    \tilde{\nu'}(\bs)~=~\begin{cases}
        \nu'(\bs) & \textrm{if $\bs\in\SetI_{\tau'}$} \\
        0 & \textrm{if $\bs\in\SetI_{\tau''}\backslash\SetI_{\tau'}$}
    \end{cases},
\end{align*}
then $\tilde{\bnu'}$ is a feasible solution of Problem~\eqref{eq:primal} with $\tau=\tau''$ because 
$\tilde{\bnu'}\matA_{\tau'',\bq}=\bnu'\matA_{\tau',\bq}\leq\blambda_{{\bq}}$, and
\begin{align}\label{eq:lemma_monotonicity:3}
    \begin{split}
    \bw_{\tau''}\cdot\tilde{\bnu'}
    &=\sum_{\bs\in\SetI_{\tau''}}\left(w(\bs)-\tau''\right)\tilde{\nu'}(\bs)
    =\sum_{\bs\in\SetI_{\tau'}}\left(w(\bs)-\tau''\right)\tilde{\nu'}(\bs)\\
    &=\sum_{\bs\in\SetI_{\tau'}}\left(w(\bs)-\tau'\right)\tilde{\nu'}(\bs)+(\tau'-\tau'')\sum_{\bs\in\SetI_{\tau'}}\tilde{\nu'}(\bs)\\
    &=\bw_{\tau'}\cdot\bnu'+(\tau'-\tau'')\|\bnu'\|_1.
    \end{split}
\end{align}
Since $\bnu''$ is an optimal solution to Problem~\eqref{eq:primal} with $\tau=\tau''$, from~\eqref{eq:lemma_monotonicity:2} and~\eqref{eq:lemma_monotonicity:3} we obtain
\begin{align*}
    \bw_{\tau'}\cdot\bnu'+(\tau'-\tau'')\|\bnu'\|_1
    &=\bw_{\tau''}\cdot\tilde{\bnu'}\\
    &\leq \bw_{\tau''}\cdot\bnu'' 
    \leq \bw_{\tau'}\cdot\bnu''_1+(\tau'-\tau'')\|\bnu''\|_1,
\end{align*}
so that
\begin{align}
  \bw_{\tau'}\cdot\bnu'+(\tau'-\tau'')\|\bnu'\|_1 
    &\leq 
    \bw_{\tau'}\cdot\bnu''_1+(\tau'-\tau'')\|\bnu''\|_1.\label{eq:lemma_monotonicity:4}
\end{align}
Then, \eqref{eq:lemma_monotonicity} and~\eqref{eq:lemma_monotonicity:4} imply  $\|\bnu'\|_1\leq\|\bnu''\|_1$ because $\tau'>\tau''$.
\end{proof}
%
Now, we prove Proposition~\ref{prop:critical threshold in W}.
We claim that any critical threshold is less than or equal to $\tau_1$ and
greater than or equal to $\tau_h$, where 
\begin{align*}
    h=\min\{k:\exists \bs\in\SetJ \textrm{ such that } w(\bs)=\tau_k,\ q_{\brho}\neq 0\, \forall\brho\in\bs \}
\end{align*}
is defined in Line~\ref{alg:check critical threshold:def of h} of Algorithm~\ref{alg:critical threshold in W}.

Since $\tau_1$ is the largest number in $\SetW$, we have $w(\bs)\leq\tau_1$ for all $\bs\in\SetI$, and thus $\bw_{\tau_1}=\bzero$.
Hence, any feasible solution in Problem~\eqref{eq:primal} with $\tau=\tau_1$ is an optimal solution to the problem.
    If Problem~\eqref{eq:primal} with $\tau=\tau_1$ has an optimal solution $\bnu$ with $\|\bnu\|\geq 1$, then $\bnu/\|\bnu\|_1$ is also an optimal solution because
    \begin{align*}
        \frac{1}{|\bnu\|_1}\bnu~=~\frac{1}{|\bnu\|_1}\bnu+\left(1-\frac{1}{|\bnu\|_1}\right)\bzero
    \end{align*}
    is a convex combination of $\bnu$ and $\bzero\in\RealSet^{\SetI_{\tau_1}}$, which is also an optimal solution. Hence, $\tau_1$ is a critical threshold.
Otherwise, all optimal solutions to the problem have $1$-norm less than $1$. Therefore, by Lemma~\ref{lemma:monotonicity}, any critical threshold should be less than $\tau_1$.
  
Let $\bnu_h\in\RealSet^{\SetI_{\tau_h}}$ be an optimal solution to Problem~\eqref{eq:primal} with $\tau=\tau_h$ and $\bs_h\in\SetJ$ such that $w(\bs_h)=\tau_h$ and $q_{\brho}\neq 0$ for all $\brho\in\bs_h$.
We denote by $\bbe\in\RealSet^{\SetI_{\tau_h}}$ the vector with $e(\bs_{h})=1$ and $e(\bs)=0$ for any $\bs\in\SetI_{\tau_h}\backslash\{\bs_h\}$.
Then, for any $\alpha\in\RealSet_+$, we have $\bnu_h+\alpha\bbe\geq \bzero$.
Moreover, for all $\brho\in\SetJ_{\bq}$, we obtain $A(\bs_h,\brho)=0$, and thus $A_{\tau_h,\bq}(\bs_h,\brho)=0$. Therefore, we have $\bbe\matA_{\tau_h,\bq}=\bzero$ so that
\begin{align*}
    \left(\bnu_h+\alpha\bbe\right)\matA_{\tau_h,\bq}
    ~=~\bnu_h\matA_{\tau_h,\bq}+\alpha\bbe\matA_{\tau_h,\bq}
    ~=~\bnu_h\matA_{\tau_h,\bq}
    ~\leq~\blambda_{{\bq}},
\end{align*}
which implies that $\bnu_h+\alpha\bbe$ is in the feasible set of Problem~\eqref{eq:primal} with $\tau=\tau_h$.
Furthermore, we obtain
\begin{align*}
    \bw_{\tau_h}\cdot(\bnu_h+\alpha\bbe)
    &=\bw_{\tau}\cdot\bnu_h+\alpha\bw_{\tau_h}\cdot\bbe\\
    &=\bw_{\tau}\cdot\bnu_h+\alpha\,w_{\tau_h}\!(\bs_h)\,e(\bs_h)
    =\bw_{\tau}\cdot\bnu_h
\end{align*}
because $w_{\tau_h}(\bs_h)=w(\bs_h)-\tau_h=0$.
Hence, $\bnu_h+\alpha\bbe_{\bs_h}$ is an optimal solution to Problem~\eqref{eq:primal} with $\tau=\tau_h$.
However, we also have
\begin{align*}
    \|\bnu_h+\alpha\bbe_{\bs_h}\|_1=\|\bnu_h\|_1+\alpha.
\end{align*}
Here $\alpha\ge 0$ can be arbitrary, so~\eqref{eq:primal} with $\tau=\tau_h$ has an optimal solution with $1$-norm greater than $1$.
Therefore, by Lemma~\ref{lemma:monotonicity}, any critical threshold at state $\bq$ is greater than or equal to $\tau_h$.

Next, note that Lines~\ref{alg:check critical threshold:start update l h}--\ref{alg:check critical threshold:end update l h} in Algorithm~\ref{alg:critical threshold in W} update $l$ and $h$
so that Problems~\eqref{eq:primal} with $\tau=\tau_l$ and $\tau=\tau_h$ have an optimal solution with $1$-norm that is less than and greater than $1$, respectively.
Hence, a critical threshold is found between $\tau_l$ and $\tau_h$ during the algorithm.

Now, assume that $\SetW$ has a critical threshold.
If $\tau_1$ or $\tau_h$ is a critical threshold, Algorithm~\ref{alg:critical threshold in W} returns $1$ or $h$ as in
Lines~\ref{alg:check critical threshold:initial check start}--\ref{alg:check critical threshold:initial check end}.
In the {\bf While} loop, $m$ is the midpoint between $l$ and $h$ and if $\tau_m$ is a critical threshold, then it is returned in Line~\ref{alg:check critical threshold:return m}.
If not, $l$ or $h$ is updated and, at each iteration, the gap between $l$ and $h$ is reduced by half as part of the binary search. 
Algorithm~\ref{alg:critical threshold in W} therefore finds a critical threshold, returning $m$ such that $\tau_m$ is the critical threshold, within a finite number of iterations.
Otherwise, the {\bf While} loop ends after a finite number of iterations and, in Line~\ref{alg:check critical threshold:return -l}, the algorithm returns the negative integer $-l$,
where any optimal solution to Problem~\eqref{eq:primal} with $\tau=\tau_l$ has $1$-norm less than $1$.
Moreover, since $h=l+1$ (from the condition in the {\bf While} loop), all optimal solutions to Problem~\eqref{eq:primal} with $\tau=\tau_h=\tau_{l+1}$ have $1$-norm greater than $1$.

\subsection{Proof of Proposition~\ref{prop:critical threshold not in W 1}}
(i) For any $\tau\in(\tau_{l+1},\tau_l)$, since there is no $\bs\in\SetI$ such that $w(\bs)\in(\tau_{l+1},\tau_{l})$, we have
\begin{align*}
   \SetI_{\tau}=\{\bs\in\SetI:w(\bs)\geq \tau\}=\{\bs\in\SetI:w(\bs)\geq \tau_l\}=\SetI_{\tau_l},
\end{align*}
and
\begin{align*}
    \bw_{\tau}\cdot\bnu
    &=\sum_{\bs\in\SetI_{\tau_l}}\left( w(\bs)-\tau\right)\nu(\bs)\\
    &=\sum_{\bs\in\SetI_{\tau_l}}w(\bs)\nu(\bs)-\tau\sum_{\bs\in\SetI_{\tau_l}}\nu(\bs)\\
    &=\bar{\bw}\cdot\bnu-\tau\|\bnu\|_1,\quad \forall\bnu\in\RealSet_+^{\SetI_{\tau}}=\RealSet_+^{\SetI_{\tau_l}}.
\end{align*}
Then, Problems~\eqref{eq:primal sub} and~\eqref{eq:primal} are equivalent, because all constraints and objective functions are the same.

(ii) From Algorithm~\ref{alg:critical threshold in W}, $\tau_l>\tau_h$ where 
    \begin{align*}
    h=\min\{k:\exists \bs\in\SetJ \textrm{ such that } w(\bs)=\tau_k,\ q_{\brho}\neq 0\ \forall\brho\in\bs \}.
    \end{align*}
    Therefore, for any $\bs\in\SetI_{\tau_l}$, there exists a $\brho\in\SetJ$ such that $\brho\in\bs$ and $q(\brho)=0$.
    If $\bnu$ is a feasible solution of Problem~\eqref{eq:primal sub}, by the constraints in Problem~\eqref{eq:primal sub}, we have for $\bs\in\SetI_{\tau_l}$ that
    $0\leq \nu(\bs)\leq \lambda(\brho)$, where $\brho\in\SetJ$ is the queue such that $\brho\in\bs$ and $q(\brho)=0$.
    In other words, the feasible region of Problem~\eqref{eq:primal sub} is bounded; namely, it is a polytope.

(iii) We prove the proposition by contradiction. Suppose that $\bnu^*$ is an optimal solution to Problem~\eqref{eq:primal sub} with $\tau=\tau_{l+1}$ such that $\|\bnu^*\|_1<1$. Define $\tilde{\bnu}^*\in\RealSet^{\SetI_{\tau_{l+1}}}$ by  $\tilde{\nu}^*(\bs)~=~\nu^*(\bs)$ if  $\bs\in\SetI_{\tau_l}$ and zero otherwise (i.e., $\bs\in\SetI_{\tau_{l+1}}\backslash\SetI_{\tau_l}$) .
Then,  $\tilde{\bnu}^*\matA_{\tau_{l+1},\bq}~=~ \bnu^*\matA_{\tau_l,\bq}~\leq~\blambda_{\bq}$,
which implies that $\tilde{\bnu}^*$ is feasible to~\eqref{eq:primal} with $\tau=\tau_{l+1}$.

On the other hand, for every feasible solution $\tilde{\bnu}$ of~\eqref{eq:primal} with $\tau=\tau_{l+1}$, if we define
$\bnu\in\RealSet^{{\SetI}_{\tau_l}}$ by $\nu(\bs)=\tilde{\bnu}(\bs)$ for $\bs\in\SetI_{\tau_l}$, we obtain 
\begin{align*}
    \bw_{\tau_{l+1}}\cdot\tilde{\bnu}
    &=\sum_{\bs\in\SetI_{\tau_{l+1}}}\left(w(\bs)-\tau_{l+1}\right)\tilde{\nu}(\bs)
    =\sum_{\bs\in\SetI_{\tau_0}}\left(w(\bs)-\tau_{l+1}\right)\tilde{\nu}(\bs) \\
    &=\sum_{\bs\in\SetI_{\tau_0}}\left(w(\bs)-\tau_{l+1}\right)\nu(\bs) 
    =\bar{\bw}\cdot\bnu-\tau_{l+1}\|\bnu\|_1\\
    &\leq\bar{\bw}\cdot\bnu^*-\tau_{l+1}\|\bnu^*\|_1
    =\bw_{\tau_{l+1}}\cdot\tilde{\bnu}^*.
\end{align*}
Therefore, $\tilde{\bnu}^*$ is an optimal solution to Problem~\eqref{eq:primal} with $\tau=\tau_{l+1}$ satisfying $\|\tilde{\bnu}^*\|_1=\|\bnu^*\|_1<1$.
By Proposition~\ref{prop:critical threshold in W}, all optimal solutions to Problem~\eqref{eq:primal} with $\tau=\tau_{l+1}$ have $1$-norm greater than $1$, which contradicts the assumption $\|\bnu^*\|_1<1$.

\subsection{Proof of Proposition~\ref{prop:critical threshold not in W}}
%
The next sequence of lemmas establishes Proposition~\ref{prop:critical threshold not in W}.
\begin{lemma}\label{lemma:critical threshold not in W}
    In Algorithm~\ref{alg:critical threshold not in W}, any optimal solution to Problem~\eqref{eq:primal sub} with $\tau=\tau^L_k$ has $1$-norm less than $1$ and any optimal solution to Problem~\eqref{eq:primal sub} with $\tau=\tau^S_k$ has $1$-norm greater than $1$, for any $k\in\IntegerSet_{+}$.  We also have, for $k\in\IntegerSet_{+}$, $\tau^M_{k+1}~\in~(\tau^S_{k+1},\tau^L_{k+1})~\subset~(\tau^S_k,\tau^L_k)$.
\end{lemma}

\begin{proof}
We prove the lemma statements by induction on $k$. 
For $k=0$, both claims are true because of the assumption of the input $l$.
Now, assume that the claims hold up until $k\geq 0$.
Then, if the condition in Line~\ref{line:condition} of Algorithm~\ref{alg:critical threshold not in W} is true, the algorithm finishes and there is nothing to prove. 
When this condition is false, suppose that $\|\bnu^M_k\|_1>1$ and then,
since $\tau^L_{k+1}=\tau^L_k$, the $1$-norm of any optimal solution to Problem~\eqref{eq:primal sub} with $\tau=\tau^L_k$ is less than $1$.
For Problem~\eqref{eq:primal sub} with $\tau=\tau^S_{k+1}=\tau^M_k$, if it has an optimal solution~$\bnu^*$ with $\|\bnu^*\|_1<1$,
we have another optimal solution 
\begin{align*}
    \left(1-\frac{1-\|\bnu^*\|_1}{\|\bnu^M_k\|_1-\|\bnu^*\|_1}\right)\bnu^*+\frac{1-\|\bnu^*\|_1}{\|\bnu^M_k\|_1-\|\bnu^*\|_1}\bnu^M_k,
\end{align*}
which is a convex combination of two optimal solutions to the problem.
Moreover, the $1$-norm of the optimal solution is
\begin{align*}
    \left(1-\frac{1-\|\bnu^*\|_1}{\|\bnu^M_k\|_1-\|\bnu^*\|_1}\right)\|\bnu^*\|_1+\frac{1-\|\bnu^*\|_1}{\|\bnu^M_k\|_1-\|\bnu^*\|_1}\|\bnu^M_k\|_1=1,
\end{align*}
which implies that $\tau^M_k$ is a critical threshold at state $\bq$ and contradicts that the condition in Line~\ref{line:condition} is false.
Hence, any optimal solution to Problem~\eqref{eq:primal sub} with $\tau=\tau^S_{k+1}$ has $1$-norm greater than $1$.
By similar arguments, the claims hold for $k+1$ when $\|\bnu^M_k\|_1<1$. 

Next, we show that $\tau^M_k\in(\tau^S_k,\tau^L_k)$ for $k\in\IntegerSet^{+}$.
For Problem~\eqref{eq:primal sub} with $\tau=\tau^L_k$, we have
\begin{enumerate}[leftmargin=0.7in]
    \item[(i)] $\bnu^L_k$ is an optimal solution;
    \item[(ii)] $\bnu^S_k$ is a feasible solution with $\|\bnu^S_k\|_1>1$; 
    \item[(iii)] No feasible solution with $1$-norm greater than $1$ is optimal;
\end{enumerate}
where the last statement is from the previous argument.
Therefore, 
\begin{align*}
&\bar{\bw}\cdot\bnu^S_k-\tau^L_k\|\bnu^S_k\|_1~<~\bar{\bw}\cdot\bnu^L_k-\tau^L_k\|\bnu^L_k\|_1
\Rightarrow\quad\frac{\bar{\bw}\cdot(\bnu^S_k-\bnu^L_k)}{\|\bnu^S_k\|_1-\|\bnu^L_k\|_1}~<~\tau^L_k.
\end{align*}
By similar arguments for Problem~\eqref{eq:primal sub} with $\tau=\tau^S_k$, we have
\begin{align*}
\tau^S_k~<~\bar{\bw}\cdot(\bnu^S_k-\bnu^L_k)/(\|\bnu^S_k\|_1-\|\bnu^L_k\|_1).
\end{align*}
Combining the last two inequalities, we conclude
\begin{align*}
    \tau^S_k~<~\tau^M_k=\bar{\bw}\cdot(\bnu^S_k-\bnu^L_k)/(\|\bnu^S_k\|_1-\|\bnu^L_k\|_1)~<~\tau^L_k.
\end{align*}

Lastly, we show that $(\tau^S_{k+1},\tau^L_{k+1})~\subset~(\tau^L_k,\tau^S_k)$.
If the condition in Line~\ref{line:condition} is true for $k$, then the algorithm stops and there is nothing to prove.
Otherwise,  either $(\tau^L_{k+1},\tau^S_{k+1})=(\tau^M_k,\tau^S_k)$ or 
$(\tau^L_{k+1},\tau^S_{k+1})=(\tau^L_k,\tau^M_k)$, all of which satisfies $(\tau^S_{k+1},\tau^L_{k+1})~\subset~(\tau^L_k,\tau^S_k)$.
\end{proof}

%

\begin{lemma}\label{lemma:critical threshold not in W:2}
    In Algorithm~\ref{alg:critical threshold not in W}, if $\bnu^L_k\neq\bnu^L_{k+1}$, then $\bnu^M_{k'}\neq\bnu^L_k$ for any $k'>k$; If $\bnu^S_k\neq\bnu^S_{k+1}$, then $\bnu^M_{k'}\neq\bnu^S_k$ for any $k'>k$.
\end{lemma}
\begin{proof}
By symmetry, we only need to prove the first statement. 
Assume that $\bnu^L_k\neq\bnu^L_{k+1}$.
Then, $\bnu^L_{k+1}=\bnu^M_k$ and $\|\bnu^M_k\|_1<1$, and $\tau^M_k$ is not a critical threshold.
We also claim that $\bnu^L_k$ is not an optimal solution to Problem~\eqref{eq:primal sub} with $\tau=\tau^M_k$.
Suppose for contradiction that it is.
From the definition of $\tau^M_k$, we obtain
\begin{align*}
    \bar{\bw}\cdot\bnu^L_k-\tau^M_k\|\bnu^L_k\|_1~=~\bar{\bw}\cdot\bnu^S_k-\tau^M_k\|\bnu^S_k\|_1,
\end{align*}
which implies that $\bnu^S_k$ is also an optimal solution to Problem~\eqref{eq:primal sub} with $\tau=\tau^M_k$.
Hence, for $\alpha=\frac{1-\|\bnu^L_k\|_1}{\|\bnu^S_k\|_1-\|\bnu^L_k\|_1}\in(0,1)$, we have that
    $(1-\alpha)\bnu^L_k+\alpha\bnu^S_k$ is an optimal solution satisfying 
    \begin{align*}
        \|(1-\alpha)\bnu^L_k+\alpha\bnu^S_k\|_1=(1-\alpha)\|\bnu^L_k\|_1+\alpha\|\bnu^S_k\|_1=1,
    \end{align*}
    which implies that $\tau^M_k$ is a critical threshold, and thus rendering a contradiction.
Hence, we prove the claim, and therefore we obtain
\begin{align}\label{eq:lemma:critical threshold not in W:1}
    \bar{\bw}\cdot\bnu^L_k-\tau^M_k\|\bnu^L_k\|_1~<~\bar{\bw}\cdot\bnu^M_k-\tau^M_k\|\bnu^M_k\|_1.
\end{align}
Moreover, by Lemma~\ref{lemma:critical threshold not in W}, we have $\tau^L_k>\tau^M_k$, and thus by Lemma~\ref{lemma:monotonicity} we obtain $\|\bnu^L_k\|_1\leq\|\bnu^M_k\|_1$.
If $\|\bnu^L_k\|_1=\|\bnu^M_k\|_1$, we have $\bar{\bw}\cdot\bnu^L_k<\bar{\bw}\cdot\bnu^M_k$ from \eqref{eq:lemma:critical threshold not in W:1}, which implies $\bar{\bw}\cdot\bnu^L_k-\tau\|\bnu^L_k\|_1~<~\bar{\bw}\cdot\bnu^M_k-\tau\|\bnu^M_k\|_1$,
    for any $\tau\in(\tau_{l+1},\tau_{l})$, thus contradicting the fact that $\bnu^L_k$ is an optimal solution to~\eqref{eq:primal sub} with $\tau=\tau^L_k$.
Hence, $\|\bnu^L_k\|_1<\|\bnu^M_k\|_1$. 

Lemma~\ref{lemma:critical threshold not in W} implies $\tau^M_{k'}<\tau^L_{k+1}=\tau^M_k$ for $k'>k$; thus, from \eqref{eq:lemma:critical threshold not in W:1},
    \begin{align*}
        \tau^M_{k'}~<~\tau^M_k~&<~(\bar{\bw}\cdot\bnu^M_k-\bar{\bw}\cdot\bnu^L_k)/(\|\bnu^M_k\|_1-\|\bnu^L_k\|_1)\nonumber\\
        \Rightarrow\qquad \bar{\bw}\cdot\bnu^L_k-\tau^M_{k'}\|\bnu^L_k\|_1~&<~\bar{\bw}\cdot\bnu^M_k-\tau^M_{k'}\|\bnu^M_k\|_1,\qquad\forall k'>k.
    \end{align*}
In other words, $\bnu^L_k$ is not an optimal solution to Problem~\eqref{eq:primal sub} with $\tau=\tau^M_{k'}$ for $k'>k$. 
Therefore, $\bnu^M_{k'}\neq \bnu^L_k$ for any $k'>k$.
\end{proof}

%
Now, we prove Proposition~\ref{prop:critical threshold not in W}.
Assume that the opposite is true: the condition in Line~\ref{line:condition} is always false so that the algorithm does not terminate. 
By Lemma~\ref{lemma:critical threshold not in W:2},
for every $k\in\IntegerSet_+$, we have $k$ basic feasible solutions (vertices) of Problem~\eqref{eq:primal sub} that cannot be $\bnu^M_k$. 
Since the number of vertices in a polytope is finite, say $K$, 
Problem~\eqref{eq:primal sub} with $\tau=\tau^M_k$ does not have a basic optimal solution, which contracts the Fundamental Theorem of Linear Programming.

\subsection{Proof of Proposition~\ref{prop:well-definedness of t_{k+1}}}
For $t\in(t_k,t_{k+1})$, from the definition of $\bq^*_t$ in Line~\ref{alg:line:definition of q_t} of Algorithm~\ref{alg:optimal_control}, we obtain $\dot{\bq^*_t}~=~\blambda-\bmu_k\,\matA
    =\blambda-\bmu^*_t\,\matA$.
Moreover, $\bq^*_t$ is continuous because $\bq^*_t$ is continuous at $t_k$ for every $k$ such that $t_{k}<\infty$.

Now, since $\bnu_k$ is a feasible solution to Problem~\eqref{eq:primal} with $\tau=\tau_k$ and $\bq=\bq^*_{t+k}$, we have for $\brho\in\SetJ_{\bq^*_{k+1}},\ t\in[t_k,t_{k+1})$:
\begin{align*}
    (\mu^*_t\,A)(\brho)~=~(\mu_k\,A)(\brho)~=~(\nu_k\,A_{\tau,\bq})(\brho)~\leq~\lambda(\brho).
\end{align*}
For $\brho\in\SetJ\backslash\SetJ_{\bq_{t_k}^*}$, if $(\mu_k\,A)(\brho)-\lambda(\brho)> 0$,
because Line~\ref{alg:definition of t_{k+1}} implies that $\left((\mu_k\,A)(\brho)-\lambda(\brho)\right)(t_{k+1}-t_{k})\leq q_{t_k}(\brho)$, we obtain
\begin{align*}
    q^*_t(\brho)=q^*_{t_k}(\brho)+(t-t_{k})\lambda(\brho)-(t-t_{k})(\mu_k\,A)(\brho)>0
\end{align*}
for $t\in[t_k,t_{k+1})$.
Therefore, we have $(\mu^*_t\,A)(\brho)\leq \lambda(\brho)$ for all $\brho\in\SetJ_{\bq_t}$ when $t\in[t_k,t_{k+1})$.
In other words, $\bmu^*_t\in\SetU(\bq^*_t)$ for all $t\in\RealSet_+$, and thus $\bmu^*_t$ is a fluid-level admissible policy.

\subsection{Proof of Theorem~\ref{thm:throughput_optimal}}

    Let $k\in\IntegerSet_+$ be such that $t_k$ is a moment at which Algorithm~\ref{alg:optimal_control} updates $\bmu_t^*$ and $\bq_k\neq \bzero$.
    Then, for $t\in[t_k,t_{k+1}]$, $\bmu^*_{t}=(\bnu_k,\bzero)\in\RealSet_+^{\SetI}$ where $\bnu_k\in\RealSet_+^{\SetI_{\tau_k}}$ is an optimal solution to 
    \begin{align*}
    \begin{split}
    \textrm{max}  \quad \bw_{\tau_k}\cdot \bnu, \quad
    \textrm{s.t.} & \quad \bnu \matA_{{\tau_k},{\bq_k}} \leq \blambda_{{\bq_k}},
    \quad \|\bnu\|_1=1,
    \quad \bnu \geq \bzero
    \end{split}
    \end{align*}
    because $\tau_k$ is a critical threshold at $\bq_k$. 
    Since $w(\brho)<\tau_k$ for any $\brho\in\SetI\backslash\SetI_{\tau_k}$,
    $\bmu^*_{t}\in\RealSet_+^{\SetI}$ is an optimal solution to 
    \begin{align}\label{eq:primal_wo_multiplier}\tag{$P_{\bq}$}
    \begin{split}
        \textrm{max}  \quad \bw\cdot \bmu, \quad
        \textrm{s.t.} & \quad \bmu \matA_{0,{\bq_k}} \leq \blambda_{{\bq_k}},
        \quad \|\bmu\|_1=1,
        \quad \bmu \geq \bzero.
    \end{split}
    \end{align}
    If the arrival rate vector $\blambda$ is inside the interior of the stability region, then by well-known results (see, e.g. \cite{ziegler2012lectures}), it is inside the polytope of the permutation matrices.
Hence, there exists a representation of $\blambda$ as a convex combination of vertices.
Meanwhile, we know that the vertices correspond to schedules in the switch, and the zero vector.
Denote this combination of schedules as $\bmu'$, under which we know that $\bmu' \matA=\blambda$.
Note that $\blambda$ being an interior point also implies that $\|\bmu' \|_1<1$, and thus we can augment $\bmu'$ to $\bmu''$ with the extra capacity assigning to queues with positive surplus. 
    Hence, there exists a feasible solution to \eqref{eq:primal_wo_multiplier} such that $\bmu''\matA\bc>\bc\cdot\blambda$ 
    and, more precisely, $\bmu''\matA\bc-\bc\cdot\blambda>c\,\varepsilon$ where $\varepsilon=1-\|\bmu'\|_1$ and $c=\min_{\brho} c_{\brho}$.
    Since $\bmu_t$ is an optimal solution to \eqref{eq:primal_wo_multiplier}, 
    \begin{align*}
        \bc\cdot\dot{\bq^*_t}=\bc\cdot\blambda-\bmu^*_t\matA\bc=\bc\cdot\blambda-\bmu^*_t\cdot\bw\leq 
        \bc\cdot\blambda-\bmu''\matA\bc<-c\,\varepsilon, 
    \end{align*}
    which implies the weighted queue length decreases at a nonzero rate until it reaches zero. 

    Next, assuming that $\bq_{T}=\bzero$, we then have the critical threshold $\tau=0$ and $\SetI_{\tau}=\SetI$.
Hence, the first part of the constraints in Problem~\eqref{eq:dual} is given by
   $ \matA\bzeta~\geq~\bw_0~=~\matA\bc$.
For every $\brho\in\SetJ$, define $\bbe_{\brho}\in\RealSet^{\SetI}_+$  by $e_{\brho}(\brho)=1$ and $e_{\brho}(\brho')=0$ if $\brho'\neq \brho$.
Then, upon multiplying $\matA\bzeta~\geq~\bw_0~=~\matA\bc$ by $\bbe_{\brho}\in\SetI$, we have
\begin{align*}
    \zeta(\brho)~=~\bbe_{\brho}\matA\bzeta~\geq~\bbe_{\brho}\matA\bc~=~c(\brho),\quad\forall\brho\in\SetJ.
\end{align*}
Therefore, the optimal solution to Problem~\eqref{eq:dual} with $\tau=0$ and $\bq=\bzero$ is $\bzeta^*=\bc$.
The complementary slackness then implies $\bmu^*\matA=\blambda$ and $\bq_t=\bzero$ for all $t\geq T$.

\subsection{Proof of Theorem~\ref{thm:optimal_result}}

We prove Theorem~\ref{thm:optimal_result} by constructing functions
$\bp_t,\bbeta_t:\RealSet_+\to\in\RealSet^{\SetJ}$ and showing that they together with $(\bq^*_t, \bmu^*_t)$ satisfy the conditions in Proposition~\ref{prop:max_principle}.
Define $\SetT:=\{t_0=0,\;t_1,\;\dots,t_K\}$ to be the set of moments at which Algorithm~\ref{alg:optimal_control} updates $\bmu^*_{t}$.
Then, from Theorem~\ref{thm:throughput_optimal}, we have that $K<\infty$ and $\bq^*_t=\bzero$ for $t\geq t_K$. Let $t_{K+1}=\infty$. 
Define Problem~\eqref{eq:dual} to be the dual of Problem~\eqref{eq:primal} given as
 \begin{align}\label{eq:dual}\tag{$D_{\bq,\tau}$}
 \min \blambda_{\bq} \cdot \bzeta, \quad s.t. \quad \qquad \matA_{\tau,\bq}\bzeta \geq \bw_{\tau}\quad \bzeta \geq \bzero,
\end{align}
where $\bzeta\in\RealSet^{\SetJ_{\bq}}$ is the vector of decision variables.
For each $k$, we fix an optimal solution $\bzeta_k\in\RealSet^{\SetJ_{\bq_{t_k}}}$ for Problem~\eqref{eq:dual} with $\tau=\tau_k$ and $\bq=\bq^*_{t_k}$,
and define $\bbeta_t$ for $t\in[t_k,t_{k+1})$ by 
\begin{align*}
    \eta_t(\brho)=\begin{cases}
        \zeta_k(\brho) & \textrm{if $\brho\in\SetJ_{\bq^*_{t_k}}$}\\
        0 & \textrm{otherwise}
    \end{cases}.
\end{align*}
Then, from the complementary slackness of primal/dual linear programming problems, we obtain the following important lemmas.

\begin{lemma}\label{lemma:slackness}
    We have $\bbeta_t\geq\bzero$ and $\bq^*_t\geq\bzero$ for $t\in\RealSet_+$. Furthermore, $\eta_t(\brho)>0$ only if $q^*_t(\brho)=0$ for $t\in\RealSet_+$ and $\brho\in\SetJ$,
which implies Condition~\ref{cond:slackness} in Proposition~\ref{prop:max_principle}: $\bq^*_t\cdot\bbeta_t=0$.
\end{lemma}

\begin{lemma}\label{lemma:optimality}
    For $\bs\in\SetI$ and $t\in[t_k, t_{k+1})$, we have $\left(A(c-\eta_t)\right)(\bs)\leq \tau_k$.
    If $\mu^*_t(\bs)>0$, then $\left(A(c-\eta_t)\right)(\bs)=\tau_k$.  In other words, we have
    \begin{align}
        \bmu\matA(\bc-\bbeta_t)~&\leq~\tau_k,\quad\forall \bmu\in\SetU, \label{eq:upperbound for muAc-eta} \\ 
        \bmu_t^*\matA(\bc-\bbeta_t)~&=~\tau_k, \label{eq:optimum for muAc-eta}
    \end{align}
    for $t\in[t_k,t_{k+1})$.
\end{lemma}

Defining $\bp_t$ for $t\in[t_k,t_{k+1})$ by 
\begin{align}\label{eq:def_p}
    \bp_t~:=~\int_t^{t_{k+1}} e^{\beta(t_{k+1}-t')}\left(\bc-\bbeta_{t'}\right) dt',
\end{align}
then Condition~\ref{cond:diff} of Lemma~\ref{lemma:maximum_principle} is satisfied.
From \eqref{eq:upperbound for muAc-eta}, for any $\bmu\in\SetU$ (i.e., $\bmu\geq 0$ and $\|\bmu\|_1=1$), we obtain
{\small \begin{align*}
    \bmu\matA\bp_t
    &=\int_t^{t_{k+1}} e^{\beta(t_{k+1}-t')}\bmu\matA\left(\bc-\bbeta_{t'}\right) dt'
    \leq\tau_k\int_t^{t_{k+1}} e^{\beta(t_{k+1}-t')} dt'.
\end{align*}}
Moreover, from \eqref{eq:optimum for muAc-eta}, we have
{\small \begin{align*}
    \bmu^*_t\matA\bp_t
    &=\int_t^{t_{k+1}} e^{\beta(t_{k+1}-t')}\bmu_t^*\matA\left(\bc-\bbeta_{t'}\right) dt'
    =\tau_k\int_t^{t_{k+1}} e^{\beta(t_{k+1}-t')} dt',
\end{align*}}
by the second part of Lemma~\ref{lemma:optimality}. 
Therefore, we obtain
\begin{align*}
    \bmu_t^*~\in~\arg\max\left\{ \bmu\matA\bp_t\,:\,\bmu\in\SetU \right\}
\end{align*}
and~\ref{cond:optimality} holds.

When $t\in[t_K,t_{K+1})$, i.e., $t\geq t_K$, we have $\bq_{t}=\bzero$, $\tau_K=0$, and $\SetI_{\tau_K}=\SetI$. 
Hence, the first constraint in Problem~\eqref{eq:dual} with $\bq=\bq_{t_K}$ and $\tau=\tau_K=0$ becomes 
\begin{align}\label{eq:constraint when tau is 0}
    \matA\bzeta~\geq~\bw_0~=~\matA\bc.
\end{align}
For every $\brho\in\SetJ$, define $\bbe_{\brho}\in\RealSet^{\SetI}_+$  by $e_{\brho}(\brho)=1$ and $e_{\brho}(\brho')=0$ if $\brho'\neq \brho$.
Then, upon multiplying \eqref{eq:constraint when tau is 0} by $\bbe_{\brho}\in\SetI$, we obtain
\begin{align*}
    \zeta(\brho)~=~\bbe_{\brho}\matA\bzeta~\geq~\bbe_{\brho}\matA\bc~=~c(\brho),\quad\forall\brho\in\SetJ.
\end{align*}
Thus, the optimal solution to Problem~\eqref{eq:dual} with $\tau=\tau_K$ and $\bq=\bq_{t_K}$ is $\bzeta_K=\bc$.
Since $\bbeta_{t'}=\bzeta_K=\bc$ for all $t'\geq t_K$, we have
\begin{align*}
    \bp_t~=~\int_t^\infty e^{\beta(t_{k+1}-t')}\left(\bc-\bbeta_{t'}\right)dt'~=~\bzero,
\end{align*}
which implies that $\lim_{t\to\infty} \bp_t\cdot(\bq^*_t-\bq_t)=0$ and~\ref{cond:endpoint} holds.

\subsubsection{Proof of Lemma~\ref{lemma:slackness}}
Assume that $t\in[t_k,t_{k+1})$ for some $k=0,1,\dots,K$.
Since $\bzeta_k$ is a feasible solution to~\eqref{eq:dual} with $\tau=\tau_k$ and $\bq=\bq^*_k$, we have $\bzeta_k\geq\bzero$ and  $\bbeta_t\geq\bzero$.
Moreover, from Proposition~\ref{prop:well-definedness of t_{k+1}}, we have $\bq^*_t\geq\bzero$.

        Now, assume that $\eta_{t}(\brho)>0$. 
        Then, we have $\zeta_k(\brho)>0$ and $\brho\in\SetJ_{\bq^*_k}$, which implies 
       $ q^*_{t_k}(\brho)~=~0$.
    On the other hand, by complementary slackness for \eqref{eq:primal} and \eqref{eq:dual}, we obtain
    \begin{align*}
        \zeta_k(\brho)\left(\lambda(\brho)-\left(\nu_k A_{\tau_k,\bq^*_{t_k}}\right)(\brho)\right)=0,
    \end{align*}
    where $\bnu_k$ is an optimal solution to \eqref{eq:primal} used in Line~\ref{alg:line:bmu} of Algorithm~\ref{alg:optimal_control}.
Since $\zeta_k(\brho)>0$, we have $\lambda(\brho)-\left(\nu_k A_{\tau_k,\bq^*_{t_k}}\right)(\brho)=0$ so that, for $t'\in[t_k,t_{k+1})$,
    \begin{align}\label{eq:slackness derivative}
        \dot{q}^*_{t'}(\brho)=\lambda(\brho)-\left(\mu^*_{t'} A\right)(\brho)=
        \lambda(\brho)-\left( \nu_k A_{\tau_k,\bq_{t_k}} \right)(\brho)=0.
    \end{align}
    From the fact $ q^*_{t_k}(\brho)~=~0$ and \eqref{eq:slackness derivative}, we conclude
    $q^*_{t'}(\brho)=0$ for $t'\in[t_k,t_{k+1})$. Therefore, $q^*_t(\brho)=0$.

\subsubsection{Proof of Lemma~\ref{lemma:optimality}}
Consider $t\in[t_k,t_{k+1})$ and $\bs\in\SetI$, and assume that $\mu^*_t(\bs)=\mu_k(\bs)>0$.
Then, we have 
$\nu_k(\bs)>0$. By complementary slackness for \eqref{eq:primal} and \eqref{eq:dual} with $\tau=\tau_k$ and $\bq=\bq_{t_k}$, we obtain
\begin{align*}
    \left(A_{\tau_k,\bq_{t_k}}\zeta_k\right) (\bs)= w_{\tau_k}(\bs)=w(\bs)-\tau_k.
\end{align*}
Hence, we conclude
\begin{align*}
    \left(A\left(c-\eta_t\right)\right)(\bs)=(Ac)(\bs)-(A\eta_t)(\bs)=w(\bs)-\left(A_{\tau_k,\bq_{t_k}}\zeta_k\right)(\bs)=\tau_k,
\end{align*}
which implies the second part of the lemma.

On the other hand, assume  that  $\mu^*_t(\bs)=0$. If $\bs\in\SetI_{\tau_k}$, we have 
\begin{align*}
    \left(A\left(c-\eta_t\right)\right)(\bs)=(Ac)(\bs)-(A\eta_t)(\bs)=w(\bs)-A_{\tau_k,\bq_{t_k}}\zeta_k(\bs)\leq\tau_k,
\end{align*}
where the last inequality follows from the constraints in~\eqref{eq:dual}.
If $\bs\not\in\SetI_{\tau_k}$, we then obtain
\begin{align*}
    \left(A\left(c-\eta_t\right)\right)(\bs)\leq (Ac)(\bs)=w(\bs)\leq\tau_k,
\end{align*}
and thus the lemma is proved.

%% file: numerical.tex
\section{Computational Experiments}\label{sec:experiments}

\begin{figure*}[t!]
	\centering
	\begin{subfigure}[t]{0.3\textwidth}
		\centering
		\includegraphics[height=1.5in]{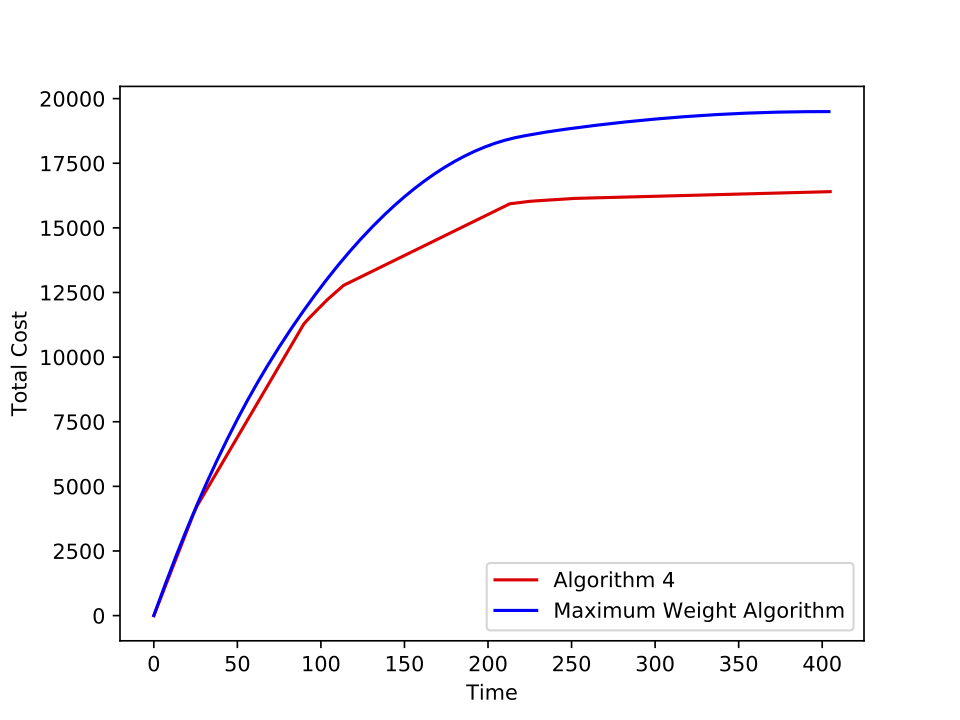}
		\caption{$\kappa=0.70$, Relative Gap is $19\%$}
	\end{subfigure}%
	~ 
	\begin{subfigure}[t]{0.3\textwidth}
		\centering
		\includegraphics[height=1.5in]{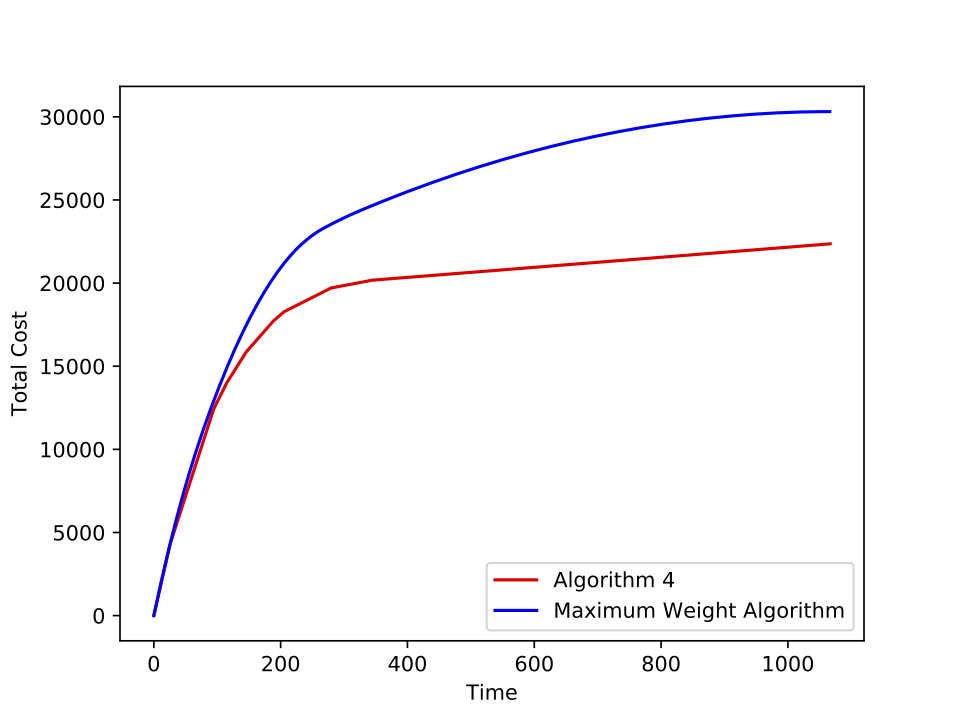}
		\caption{$\kappa=0.90$, Relative Gap is $35\%$}
	\end{subfigure}
	~
	\begin{subfigure}[t]{0.3\textwidth}
		\centering
		\includegraphics[height=1.5in]{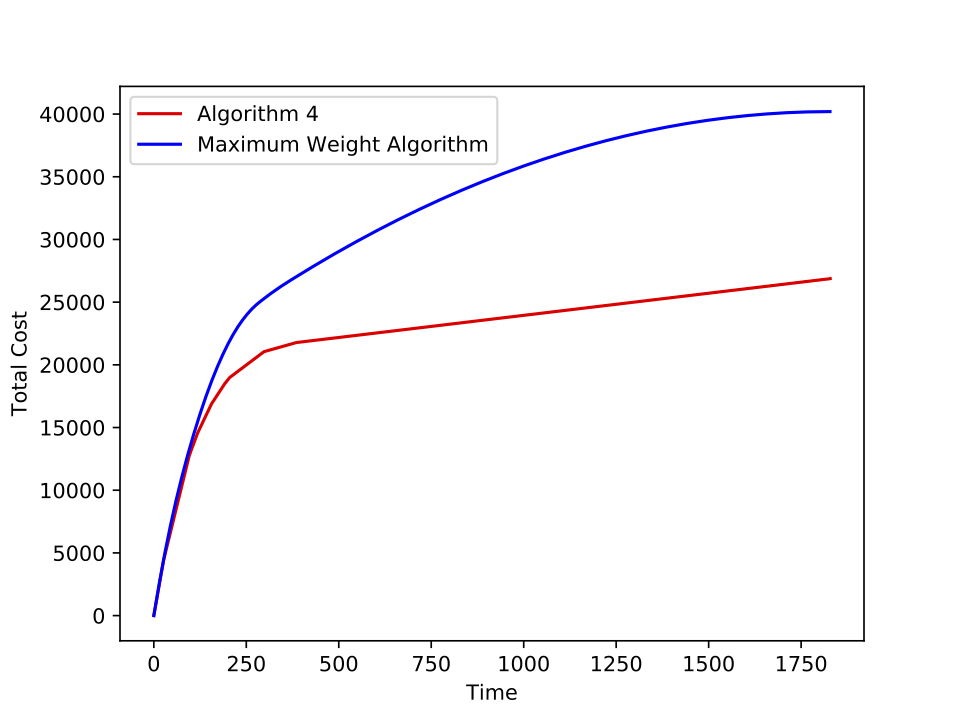}
		\caption{$\kappa=0.95$, Relative Gap is $50\%$}
	\end{subfigure}
	\caption{Performance Comparisons of Total Costs under Optimal Policy (Algorithm~\ref{alg:optimal_control}) and Max-Weight Algorithm}\label{fig:comparison}
\end{figure*}

\begin{figure}[h]
	\centering
	\includegraphics[width=0.65\linewidth]{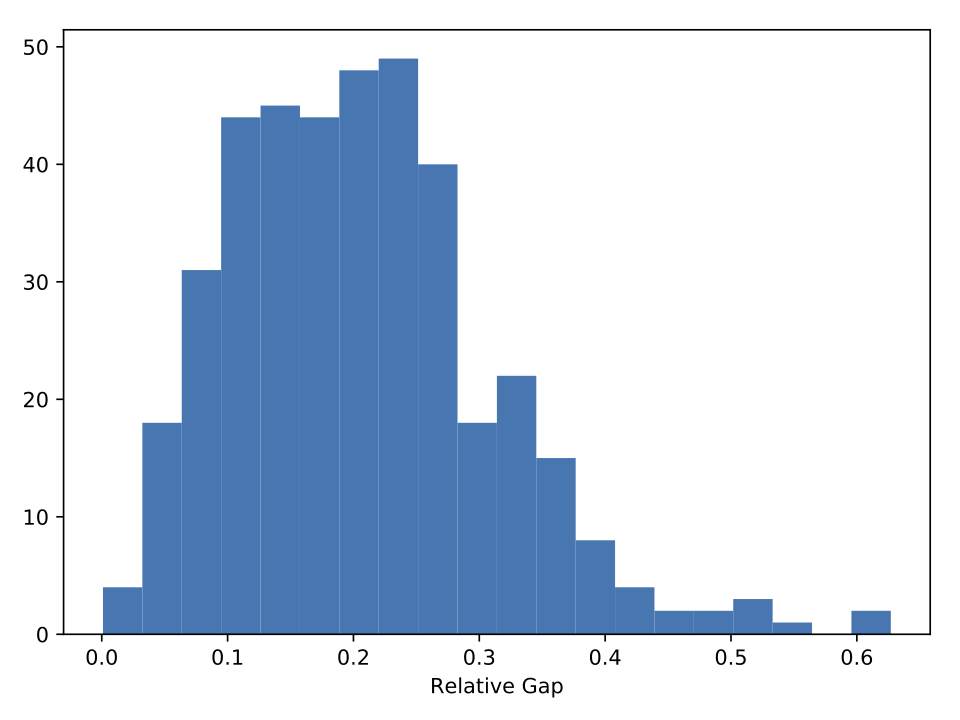}
	\caption{Histogram of Relative Gaps for $\kappa=0.9$}\label{fig:histogram}
\end{figure}

\begin{figure*}[t!]
	\centering
	\begin{subfigure}[t]{0.3\textwidth}
		\centering
		\includegraphics[height=1.5in]{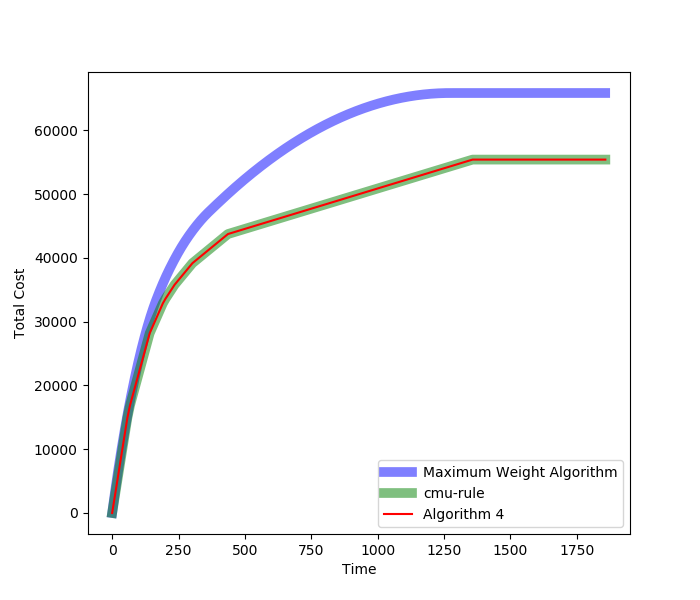}
		\caption{Algorithm 4 coincides $c\mu$-rule}\label{fig:optimal_cmu_all}
	\end{subfigure}%
	~ 
	\begin{subfigure}[t]{0.3\textwidth}
		\centering
		\includegraphics[height=1.5in]{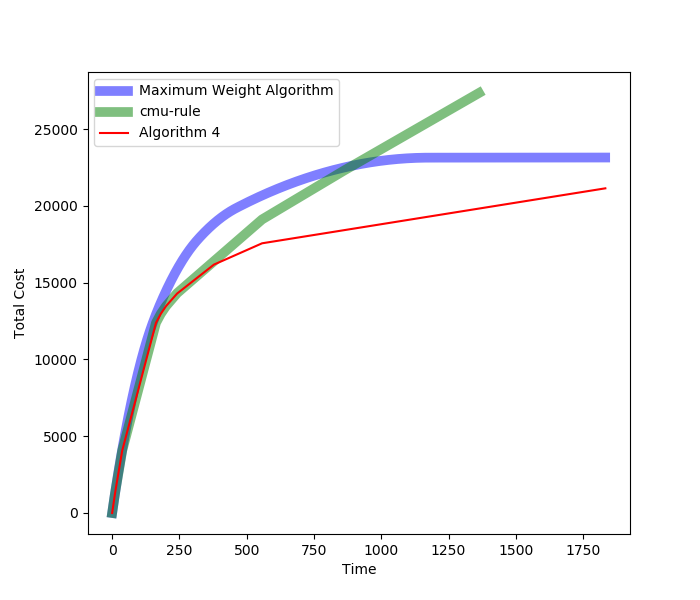}
		\caption{Unstable $c\mu$-rule}\label{fig:unstable_cmu_all}
	\end{subfigure}
	~
	\begin{subfigure}[t]{0.3\textwidth}
		\centering
		\includegraphics[height=1.5in]{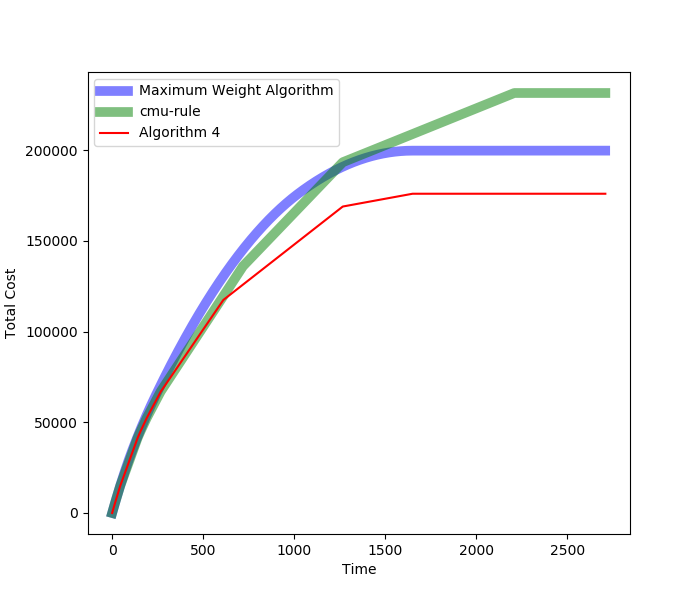}
		\caption{Stable but not optimal $c\mu$-rule}\label{fig:cmu_stable_cmu_all}
	\end{subfigure}
	\caption{Performance Comparisons of Total Costs under Optimal Policy (Algorithm~\ref{alg:optimal_control}) and $c\mu$-rule}\label{fig:cmu_comparison}
\end{figure*}

In this section, we present computational experiments that compare the performance of our optimal control algorithm with that of the max-weight scheduling algorithm and the $c\mu$ rule in the fluid model context.
We fix the number of input and output ports to be $n\in\IntegerSet_+$ and fix the throughput $\kappa\in(0,1)$.
For $1\leq i,j\leq n$, we randomly generate the costs $c(i,j)\in(0,1)$ and the arrival rates $\lambda(i,j)\in(0,1)$ such that
\begin{align}\label{eq:throughput in simulation}
  \max\left\{\sum_{k=1}^n \lambda(i,k),\sum_{k=1}^n \lambda(k,j)\,:\,i,j\in[n]\right\}~=~\kappa.
\end{align}
We also choose an initial queue length to be an integer between $1$ and $100$ uniformly at random for each $(i,j)\in[n]\times[n]$. 

With these parameters, we apply Algorithm~\ref{alg:optimal_control} until we reach the time $T$ at which the queue length becomes $0$ for all queues. 
During our experiments, we let $t_0,t_1,\dots,t_K$ denote the epochs at which Algorithm~\ref{alg:optimal_control} updates the admissible schedule, with $t_0=0$ and $t_K=T$.
Then, the total cost $ \int_0^{\infty} \bc\cdot \bq_t\,dt$ is given by
\begin{align}\label{eq:total_cost}
\begin{split}
 \sum_{k=0}^{K-1}\int_{t_k}^{t_{k+1}}\bc\cdot\bq_t dt  =\sum_{k=0}^{K-1} \bc\cdot\left(\frac{\bq_{t_{k+1}}+\bq_{t_{k}}}{2}\right)(t_{k+1}-t_k)
\end{split}
\end{align}
because on the interval $[t_k,t_{k+1}]$ the admissible schedule does not change and $\bq_t$ is a linear function.
Note that, even though the objective function in the optimal control problem~\eqref{eq:optimal control problem} has a discount factor $\beta\in(0,1)$, 
we set $\beta=1$ for the results of our computational experiments herein because Algorithm~\ref{alg:optimal_control} does not depend on $\beta$.

While the existence and uniqueness of the fluid limit under the max-weight scheduling algorithm has been proven (see~\cite{DaiPra00} and~\cite{shah2012}), 
an explicit formula is not known. 
Hence, to numerically compute the max-weight scheduling algorithm in the fluid model, we partition the interval $[0,T]$ into slots of size $\Delta t$;
then, for time slot $[t'_k, t'_k+\Delta t]$, we find a basic schedule of the max-weight algorithm with respect to $\bq_{t_k}$, say $\bs\in\SetI$,
and use this schedule during that time slot.
In other words, we set $$q_{t_{k+1}}(i,j)~=~\max\left\{ q_{t_k}(i,j)+\left(\lambda(i,j)-s(i,j)\right)\Delta t,0\right\}$$
for $(i,j)\in[n]\times[n]$ and approximately measure the total cost on the interval $[0,T]$ by (assuming that $t'_{K'}=T$)
$\int_0^T \bc\cdot \bq_t\,dt\approx\sum_{k=1}^{K'-1} \bc\cdot \bq_{t'_k}$,
which is close to the actual total cost under the max-weight scheduling algorithm as $\Delta t\to 0$ and we selected $\Delta t$ accordingly.

    ~ 
    ~

Figure~\ref{fig:comparison} illustrates a representative sample of the total cost over time on $[0,T]$ for the $3\times 3$ input-queued switch fluid model
under our optimal control policy and the max-weight scheduling policy.
The cost coefficients and the initial queue lengths are set to be the same in each of these three experiments.
We vary the throughput $\kappa$, defined in~\eqref{eq:throughput in simulation}, across the three experiments (i.e., $\kappa = 0.7, 0.9, 0.95$) while fixing the ratio among the arrival rates.
As observed in the figure, the performance of our optimal policy (Algorithm~\ref{alg:optimal_control}) improves in comparison with that of the max-weight scheduling algorithm as the throughput $\kappa$ increases.
To quantify this performance comparison, we calculate the \emph{relative gap} defined by the difference between the total costs at time $T$ under the two algorithms divided by the total cost
at time $T$ of the optimal algorithm.
The growth in this relative performance gap as the throughput increases ranges from $19\%$ for $\kappa=0.7$, to $35\%$ for $\kappa=0.9$ and $50\%$ for $\kappa=0.95$.

Figure~\ref{fig:histogram} illustrates a representative sample of the corresponding relative performance gap results for various combinations of costs, initial state, and arrival rates under a fixed
throughput of $\kappa=0.9$.
We observe that the distribution of the relative gap demonstrates improved performance of at least $10\%$, in most cases, under Algorithm~\ref{alg:optimal_control}
in comparison with the max-weight scheduling.
The sample average of the relative performance gap is around $20\%$.

We also compare the total cost under our optimal policy (Algorithm 4) and the $c\mu$-rule. 
Figure~\ref{fig:cmu_comparison} illustrates a representative sample of the total cost over time on $[0,T]$ for the $3\times 3$ input-queued switch fluid model,
demonstrating three different types of behavior.
In Figure~\ref{fig:optimal_cmu_all}, the $c\mu$-rule and the optimal algorithm are identical and provide the same performance.
We observe in Figure~\ref{fig:unstable_cmu_all}, however, that the $c\mu$-rule is unstable and clearly not optimal. 
Moreover, even when the $c\mu$-rule is stable, it may not be optimal as shown in Figure~\ref{fig:cmu_stable_cmu_all}.
The highest relative performance improvement of our optimal policy over instances of the stable $c\mu$-rule is more than $70\%$.